\documentclass[11pt,reqno]{amsart}
\usepackage{tikz}
\usepackage{subfig}
\usepackage{epstopdf}
\usepackage{epsfig}
\usepackage{math}
\graphicspath{{./fig/}}
\usepackage{tikz}
\usetikzlibrary{shapes,arrows}
\tikzstyle{decision} = [diamond, draw, fill=blue!20, 
    text width=4.5em, text badly centered, node distance=3cm, inner sep=0pt]
\tikzstyle{block} = [rectangle, draw, fill=blue!20, 
    text width=5em, text centered, rounded corners, minimum height=4em]
\tikzstyle{line} = [draw, -latex']
\tikzstyle{cloud} = [draw, ellipse,fill=red!20, node distance=3cm,
    minimum height=2em]
    
\usetikzlibrary{positioning}
\tikzset{main node/.style={circle,fill=blue!20,draw,minimum size=1cm,inner sep=0pt},  }

\usepackage{soul}
\usepackage{color}

\usepackage{hyperref}
\usepackage{graphicx} \usepackage{placeins}

\begin{document}
\title[Natural gradient via optimal transport]{Natural gradient via optimal transport}
\author[Li]{Wuchen Li}
\address{Department of Mathematics, University of California, Los Angeles, USA.}
\email{wcli@math.ucla.edu}
\author[Mont\'ufar]{Guido Mont\'ufar}
\address{Department of Mathematics and Department of Statistics, University of California, Los Angeles, USA; Max Planck Institute for Mathematics in the Sciences, Leipzig, Germany.}
\email{montufar@math.ucla.edu}

\newcommand{\vr}{\overrightarrow}
\newcommand{\wt}{\widetilde}
\newcommand{\dd}{\mathcal{\dagger}}
\newcommand{\ts}{\mathsf{T}}

\keywords{Optimal transport; Information geometry; Wasserstein statistical manifold; Displacement convexity; Machine learning.}

\maketitle

\begin{abstract}
We study a natural Wasserstein gradient flow on manifolds of probability distributions with discrete sample spaces. 
We derive the Riemannian structure for the probability simplex from the dynamical formulation of the Wasserstein distance on a weighted graph. 
We pull back the geometric structure to the parameter space of any given probability model, which allows us to define a natural gradient flow there. 
In contrast to the natural Fisher-Rao gradient, the natural Wasserstein gradient incorporates a ground metric on sample space. We illustrate the analysis of elementary exponential family examples and demonstrate an application of the Wasserstein natural gradient to maximum likelihood estimation. 
\end{abstract}

\section{Introduction}
The statistical distance between histograms plays a fundamental role in statistics and machine learning. 
It provides the {geometric} structure on statistical manifolds~\cite{IG}. 
Learning problems usually {correspond to} minimizing a loss function over these manifolds. 
An important example is the Fisher-Rao metric on the probability simplex, which has been studied especially within the field of information geometry~\cite{IG,IG2}. 
A classic result due to Chentsov~\cite{cencov1982} characterizes this Riemannian metric as the only one, up to scaling, that is invariant with respect to natural statistical embeddings by Markov morphisms (see also~\cite{Campbell,Lebanon05,e16063207}). 
Using the Fisher-Rao metric, a natural Riemannian gradient descent method is introduced~\cite{NG}. 
This natural gradient has found numerous successful applications in machine learning (see, e.g.,~\cite{NIPS1996_1248,10.1007/11564096_29,Yi:2009:SSU:1553374.1553522,Malago:2011:TGE:1967654.1967675,Pascanu+Bengio-ICLR2014}). 

Optimal transport provides {another} statistical distance, named Wasserstein {or} Earth Mover's distance. In recent years, this metric has attracted increasing attention within {the machine learning} community~\cite{WGAN, LWL, Boltzman}. 
One distinct feature of optimal transport is that it provides a distance among histograms that {incorporates} {a ground metric} on sample space. 
The $L^2$-Wasserstein distance has a dynamical formulation, which exhibits a metric tensor structure.  
The {set of probability densities} with this metric forms an infinite-dimensional Riemannian manifold, named density manifold~\cite{Lafferty}. 
The gradient descent method in the density manifold, called Wasserstein gradient flow, has been widely studied in the literature; see~\cite{otto2001, vil2008} and references. 

A question intersecting optimal transport and information geometry arises: What is the natural Wasserstein gradient descent method on the parameter space of a statistical model? In optimal transport, the Wasserstein gradient flow is studied on the full space of probability densities, and shown to have deep connections with the ground metrics on sample space deriving from physics \cite{qf}, fluid mechanics \cite{Gangbo2} and differential geometry \cite{Lott}. 
We expect that these relations also exist on parametrized probability models, and that the Wasserstein gradient flow can be useful in the optimization of objective functions that arise in machine learning problems. 
By incorporating a ground metric on sample space, this method can serve to implement useful priors in the learning algorithms. 

We are interested in developing synergies between the information geometry and optimal transport communities. 
In this paper, we take a natural first step in this direction. 
We introduce the Wasserstein natural gradient flow on the parameter space of probability models with  discrete sample spaces. 
The $L^2$-Wasserstein metric on discrete states {was introduced} in~\cite{chow2012, Maas, M}. Following the settings from~\cite{Li2, Li1, Li3, LiG}, the probability simplex forms the Riemannian manifold called Wasserstein probability manifold. The Wasserstein metric on the probability simplex can be pulled back to the parameter space of a probability model. 
This metric allows us to define a natural Wasserstein gradient method on parameter space. 

We note  that one finds several formulations of optimal transport for continuous sample spaces. On the one hand, there is the static formulation, known as Kantorovich's linear programming~\cite{vil2008}. 
Here, the linear program is to find the minimal value of a functional over the set of joint measures with given marginal histograms. 
The objective functional is given as the expectation value of the ground metric with respect to a joint probability density measure. 
On the other hand, there is the dynamical formulation, known as the Benamou-Brenier formula~\cite{Benamou2000}. 
This dynamic formulation gives the metric tensor for measures by lifting the ground metric tensor of sample spaces. Both static and dynamic formulations are equivalent in the case of continuous state spaces. However, the two formulations lead to different metrics in the simplex of discrete probability distributions. The major reason for this difference is that the discrete sample space is not a length space.\footnote{A length space is one in which the distance between points can be measured as the infimum length of continuous curves between them.} Thus the equivalence result in classical optimal transport is no longer true in the setting of discrete sample spaces. We note that for the static formulation, there is no Riemannian metric tensor for the discrete probability simplex. See~\cite{Li1, Maas} for a detailed discussion. 

In the literature, the exploration of connections between optimal transport and information geometry was initiated in~\cite{LP, IGWD, Wong}. 
These works focus on the distance function induced by linear programming on discrete sample spaces. 
As we pointed out above, this approach can not cover the Riemannian and differential structures induced by optimal transport.  
In this paper, we use the dynamical formulation of optimal transport to define a Riemannian metric structure for general statistical manifolds. 
With this, we obtain a natural gradient operator, which can be applied to any optimization problem over a parameterized statistical model. In particular, it is applicable to maximum likelihood estimation. Other works have studied the Gaussian family of distributions with $L^2$-Wasserstein metric \cite{IGW, GW}. 
In that particular case, the constrained optimal transport metric tensor can be written explicitly and the corresponding density submanifold is a totally geodesic submanifold. In contrast to those works, our discussion is applicable to arbitrary parametric models. 

This paper is organized as follows. 
In Section~\ref{sec2} we briefly review the Riemannian manifold structure in probability space introduced by optimal transport in the cases of continuous and discrete sample spaces. 
In Section~\ref{sec3} we introduce Wasserstein statistical manifolds by isometric embedding into the probability manifold, and in Section~\ref{secnew} we derive the corresponding gradient flows. 
In Section~\ref{sec4} we discuss a few examples. 

\section{Optimal transport on continuous and discrete sample spaces} 
\label{sec2}

In this section, we briefly review the results of optimal transport. We introduce the corresponding Riemannian structure for simplices of probability distributions with discrete support. 

\subsection{Optimal transport on continuous sample space}

 We start with a review of the optimal transport problem on continuous spaces. This will guide our discussion of the discrete state case.  For related studies, we refer the reader to~\cite{Lafferty, vil2008} and the many references therein. 

Denote the sample space by $(\Omega, g^\Omega)$. Here $\Omega$ is a finite dimensional smooth Riemannian manifold, for example, $\mathbb{R}^d$ or the open unit ball therein. 
Its inner product is denoted by $g^\Omega$ and its volume form by $dx$. Denote the geodesic distance of $\Omega$ by $d_\Omega\colon \Omega\times \Omega\rightarrow \mathbb{R}_+$.  

Consider the set $\mathcal{P}_2(\Omega)$ of {Borel measurable probability density functions} on $\Omega$ with finite second moment. 
Given {$\rho^0, \rho^1\in \mathcal{P}_2(\Omega)$, the $L^2$-Wasserstein distance {between $\rho^0$ and $\rho^1$ is denoted by $W\colon \mathcal{P}(\Omega)\times \mathcal{P}(\Omega)\rightarrow \mathbb{R}_+$. There are two equivalent ways of defining this distance. On one hand, there is the static formulation. This refers to the following linear programming problem: 
\begin{equation}
\label{eq:static}
W(\rho^0, \rho^1)^2=\inf_{\pi\in \Pi(\rho^0, \rho^1)}\int_{\Omega\times \Omega}d_\Omega(x,y)^2\pi(dx,dy),
\end{equation}
where the infimum is taken over the set $\Pi(\rho^0, \rho^1)$ {of joint probability measures} {on} $\Omega\times \Omega$ that have marginals $\rho^0$, $\rho^1$. 

On the other hand, the Wasserstein distance $W$ can be written in a dynamic formulation, where a probability path $\rho:[0,1]\rightarrow \mathcal{P}_2(\Omega)$ connecting $\rho^0$, $\rho^1$ is considered. This refers to a variational problem known as the Benamou-Brenier formula: 
\begin{subequations}\label{BB1}
\begin{equation}\label{BB}
W(\rho^0, \rho^1)^2=\inf_{\Phi}~\int_0^1\int_{\Omega} g^\Omega_x(\nabla\Phi(t,x), \nabla\Phi(t,x))\rho(t,x) dx dt, 
\end{equation}
where the infimum is taken over the set of Borel \emph{potential} functions {$[0,1]\times \Omega \to \mathbb{R}$.} 
Each potential function $\Phi$ determines a corresponding density path $\rho$ as the solution of the  \emph{continuity equation}  
\begin{equation}\label{BB2}
\frac{\partial \rho(t,x)}{\partial t}+\textrm{div} (\rho(t,x)\nabla\Phi(t,x))=0,\quad \rho(0,x)=\rho^0(x),\quad \rho(1,x)=\rho^1(x).
\end{equation}
\end{subequations}
Here $\textrm{div}$ and $\nabla$ are the divergence and gradient operators in $\Omega$. 
The continuity equation is well known in physics. 

The equivalence of the static~\eqref{eq:static} and dynamic~\eqref{BB1} formulations is well known (for continuous $\Omega$). 
For the reader's convenience we give a sketch of proof in the appendix. 
{In this paper we focus on the variational formulation~\eqref{BB1}.} 
In fact, this formulation {entails the definition of a Riemannian structure as we now discuss.} 
For simplicity, we only consider {the set of smooth and strictly positive probability densities}
\begin{equation*}
\mathcal{P}_+(\Omega)=\Big\{\rho \in C^{\infty}(\Omega)\colon \rho(x)>0,~\int_{\Omega}\rho(x)dx=1\Big\} \subset \mathcal{P}_2(\Omega).
\end{equation*}
Denote $ \mathcal{F}(\Omega):=C^{\infty}(\Omega)$ the set of smooth real valued functions on $\Omega$. 
The tangent space of $\mathcal{P}_+(\Omega)$ is given by 
$$
T_\rho\mathcal{P}_+(\Omega) = \Big\{\sigma\in  \mathcal{F}(\Omega)\colon \int_{\Omega}\sigma(x) dx=0 \Big\}.
$$
Given $\Phi\in  \mathcal{F}(\Omega)$ and $\rho\in \mathcal{P}_+(\Omega)$, define
\begin{equation*}
V_{\Phi}(x):=-\textrm{div} (\rho(x) \nabla \Phi(x)). 
\end{equation*}
We assume the zero flux condition 
\begin{equation*}
\int_{\Omega}V_\Phi(x)dx 
=0. 
\end{equation*}
In view of the continuity equation, the zero flux condition is equivalent to requiring that $\int_\Omega \frac{\partial \rho}{\partial t}dx = 0$, which means that the space integral of $\rho$ is always $1$. 
When $\Omega$ is compact without boundary, this is automatically satisfied. This is also true when $\Omega = \mathbb{R}^d $ and $\rho$ has finite second moment. 
Thus $V_\Phi\in T_{\rho}\mathcal{P}_+(\Omega)$. 
The elliptic operator $\nabla\cdot(\rho\nabla)$ identifies the function $\Phi$ on $\Omega$ modulo additive constants with a tangent vector $V_{\Phi}$ of the space of densities (for more details see~\cite{Lafferty, Lott}). 
This gives an isomorphism 
$$
 \mathcal{F}(\Omega)/\mathbb{R}\rightarrow T_{\rho}\mathcal{P}_+(\Omega); \quad \Phi \mapsto V_\Phi .
$$ 
Define the Riemannian metric (inner product) on the tangent space of positive densities $g^W\colon {T_\rho}\mathcal{P}_+(\Omega)\times {T_\rho}\mathcal{P}_+(\Omega)\rightarrow \mathbb{R}$ by
\begin{equation*}
g^W_\rho(V_{\Phi}, V_{\tilde{\Phi}})=\int_{\Omega}g^\Omega_x(\nabla\Phi(x), \nabla\tilde{\Phi}(x))\rho(x) dx,
\end{equation*}
where $\Phi(x)$, $\tilde{\Phi}(x)\in  \mathcal{F}(\Omega)/\mathbb{R}$. 
This inner product endows $\mathcal{P}_+(\Omega)$ with an infinite dimensional Riemannian metric tensor. 
In other words, the variational problem~\eqref{BB1} is a geometric action energy in $(\mathcal{P}_+(\Omega), g^W)$ in the sense of~\cite{Benamou2000, Lott}.  
In literature~\cite{Lafferty}, $(\mathcal{P}_+(\Omega), g^W)$ is called density manifold. 

\subsection{Dynamical optimal transport on discrete sample spaces}

We translate the dynamical perspective from the previous section to discrete state spaces, i.e., we replace the continuous space $\Omega$ by a discrete space $I=\{1,\ldots, n\}$. 

To encode the metric tensor of discrete states, we first need to introduce a ground metric notion on sample space. We do this in terms of a graph with weighted edges, $G=(V, E, \omega)$, where
$V=I$ is the vertex set, $E$ is the edge set, and $\omega=(\omega_{ij})_{i,j\in I}\in \mathbb{R}^{n\times n}$ are the edge weights. 
These weights satisfy 
$$
\omega_{ij}=
\begin{cases}
\omega_{ji}>0, & \textrm{if $(i,j)\in E$}\\
0, & \textrm{otherwise}
\end{cases}.
$$ 
As mentioned above, the weights encode the ground metric on the discrete state space. 
More precisely, we write 
\begin{equation}\label{GM}
\omega_{ij}=\frac{1}{(d^G_{ij})^2},\quad\textrm{if $(i,j)\in E$},
\end{equation}
where $d^G_{ij}$ represents the distance or \emph{ground metric} between states $i$ and $j$. 
The set of neighbors or adjacent vertices of $i$ is denoted by $N(i)=\{j\in V\colon (i,j)\in E\}$. 

The probability simplex supported on the vertices of $G$ is defined by 
\begin{equation*}
\mathcal{P}(I) = \Big\{(p_1,\ldots, p_n)\in \mathbb{R}^n \colon \sum_{i=i}^n p_i=1,\quad  p_i\geq 0\Big\}. 
\end{equation*}
Here $p=(p_1,\ldots, p_n)$ is a probability vector with coordinates $p_i$ corresponding to the probabilities assigned to each node $i\in I$. 
We denote the relative interior of the probability simplex by $\mathcal{P}_+(I)$. This consists of the strictly positive probability distributions, $p\in\mathcal{P}(I)$ with $p_i>0$, $i\in I$. 

Next we introduce the variational problem~\eqref{BB1} on discrete states. 
First we need to define the ``metric tensor'' on graphs. 
A {\em vector field} $v=(v_{ij})_{i,j\in V}\in \mathbb{R}^{n\times n}$ on $G$ is a  
{skew-symmetric matrix}:  
\begin{equation*}
v_{ij}=\begin{cases}-v_{ji}, & \textrm{if $(i,j)\in E$}\\
0, & \textrm{otherwise}
\end{cases}. 
\end{equation*} 
A potential function $\Phi=(\Phi_i)_{i=1}^n\in\mathbb{R}^{n}$ defines a {\em gradient vector field} $\nabla_G\Phi=(\nabla_G\Phi_{ij})_{i,j\in V}\in \mathbb{R}^{n\times n}$ on the graph $G$ by the finite differences 
\begin{equation*}
\nabla_G\Phi_{ij}=\begin{cases}\sqrt{\omega_{ij}}(\Phi_i-\Phi_j) & \textrm{if $(i,j)\in E$}\\
0 & \textrm{otherwise}
\end{cases}.
\end{equation*}
Here we use $\sqrt \omega$ rather than $1/d^G$ for simplicity of notations. In this way, we can represent the gradient, divergence, and Laplacian matrix in a multiplicity of weight, instead of dividing the ground metric. 

We define an inner product of vector fields $v_{ij}$, $\tilde{v}_{ij}$ at each state $i \in I$ by 
\begin{equation*}
g^I_i(v, \tilde v) := 
\frac12 \sum_{j\in N(i)}v_{ij}\tilde{v}_{ij}.
\end{equation*}
In particular, the gradient vector field $\nabla_G\Phi $ defines a kinetic energy at each state $i \in I$ by  
$$
g^I_i(\nabla_G\Phi, \nabla_G\Phi) := 
\frac12 \sum_{j\in N(i)} (\Phi_i -\Phi_j)^2 \omega_{ij}. 
$$

We next define the expectation value of kinetic energy with respect to a probability distribution $p$: 
\begin{equation*}
\begin{split}
(\nabla_G\Phi,\nabla_G \Phi)_p
: =& 
\sum_{i\in I} p_i\;
g^I_i(\nabla_G\Phi, \nabla_G\Phi)=\frac{1}{2}\sum_{(i,j)\in E}\omega_{ij}(\Phi_i-\Phi_j)^2\frac{p_i+p_j}{2}. 
 \end{split}
 \end{equation*}
This can also be written as 
\begin{equation*}
(\nabla_G\Phi,\nabla_G \Phi)_p=\sum_{i=1}^n\Phi_i\sum_{j\in N(i)}{\omega_{ij}}(\Phi_i-\Phi_j)\frac{p_i+p_j}{2}=\Phi^{\ts}\big(-\textrm{div}_G(p\nabla_G\Phi)\big),
\end{equation*}
where
\begin{equation}\label{div}
-\textrm{div}_G(p \nabla_G\Phi):=
\Bigl(\sum_{j\in N(i)}\omega_{ij}(\Phi_i-\Phi_j)\frac{p_i+p_j}{2}\Bigr)_{i\in I}.
\end{equation}
There are two definitions hidden in~\eqref{div}. 
First, 
$\operatorname{div}_G\colon \mathbb{R}^{n\times n}\rightarrow\mathbb{R}^n$ maps any given vector field $m$ on the graph $G$ to 
a potential function 
\begin{equation*}
\textrm{div}_G(m)=\big(\sum_{j\in N(i)}\sqrt{\omega_{ij}}m_{ji}\big)_{i\in I}.
\end{equation*}
Second, the probability weighted  gradient vector field $m=p\nabla_G\Phi$ defined by
\begin{equation*}
m_{ij}=
\begin{cases}
\frac{p_i+p_j}{2}(\Phi_i-\Phi_j)\sqrt{\omega_{ij}}, & \textrm{if $(i,j)\in E$}\\
0, & \textrm{otherwise}
\end{cases},
\end{equation*}
where $\frac{p_i+p_j}{2}$ represents the probability weight on the edge $(i,j)\in E$. 

We are now ready to introduce the $L^2$-Wasserstein metric on $\mathcal{P}_+(I)$. 
\begin{definition}\label{def_metric}
For any $p^0$, $p^1\in \mathcal{P}_+(I)$, define the Wasserstein distance $W\colon\mathcal{P}_+(I)\times\mathcal{P}_+(I)\rightarrow\mathbb{R}$ by 
\begin{equation*}
W(p^0,p^1)^2:= \inf_{p(t), \Phi(t)}~\Big\{\int_0^1(\nabla_G \Phi(t), \nabla_G \Phi(t))_{p(t)} dt\Big\}.
\end{equation*}
Here the infimum is taken over pairs $(p(t), \Phi(t))$ with 
$p\in H^1((0,1), \mathbb{R}^{n})$ and $\Phi\colon
[0, 1]\rightarrow\mathbb{R}^n$ measurable, satisfying 
\begin{equation*}
\dot p(t)+\operatorname{div}_G(p(t) \nabla_G\Phi(t))=0,~ p(0)=p^0,~p(1)=p^1. 
\end{equation*}
\end{definition}

\begin{remark}
It is worth mentioning that the metric given in Definition~\ref{def_metric} is {different} from the metric defined by linear programming. 
In other words, denote the distance $d^G(i,j)$ between
two vertices $i$ and $j$ as the length of a shortest $(i,j)$-path. 
If $(i,j)\in E$, then $d^G(i,j)$ is same as the ground metric defined in \eqref{GM}. 
Then
 \begin{equation}\label{linear}
\big(W(p^0, p^1)\big)^2
\not\equiv \min_{\pi} \Big\{ \sum_{1\leq i,j\leq n} d_G(i,j)^2 \pi_{ij} ~:~\sum_{i=1}^n\pi_{ij}=p^0_j ,\quad \sum_{j=1}^n\pi_{ij}=p^1_i,\quad \pi_{ij}\geq 0 \Big\}.
\end{equation}
The reason for this in-equivalence is that the discrete sample space $I$ is not a length space. 
In other words, there is no continuous path in $I$ connecting two nodes in $I$. 
For more details see discussions in the appendix.
\end{remark}
\subsection{Wasserstein geometry and discrete probability simplex}

In this section we introduce the primal coordinates of the discrete probability simplex with $L^2$-Wasserstein Riemannian metric. 
Our discussion follows the recent work~\cite{LiG}. 
The probability simplex $\mathcal{P}(I)$ is a manifold with boundary. 
To simplify the discussion, we focus on the interior $\mathcal{P}_+(I)$. 
The geodesic properties on the boundary $\partial\mathcal{P}(I)$ have been studied in~\cite{Li3}. 

Let us focus on the Riemannian structure. 
In the following we introduce an inner product on the tangent space 
\begin{equation*}
T_p\mathcal{P}_+(I) = \Big\{(\sigma_i)_{i=1}^n\in \mathbb{R}^n\colon  \sum_{i=1}^n\sigma_i=0 \Big\}.
\end{equation*}
Denote the space of potential functions on $I$ by $ \mathcal{F}(I)=\mathbb{R}^{n}$. 
Consider the quotient space 
\begin{equation*}
 \mathcal{F}(I)/ \mathbb{R}=\{[\Phi] \colon (\Phi_i)_{i=1}^n\in \mathbb{R}^n\},
\end{equation*}
where $[\Phi]=\{(\Phi_1+c,\ldots, \Phi_n+c) \colon c\in\mathbb{R}\}$ are functions defined up to addition of constants. 

We introduce an identification map via~\eqref{div}
\begin{equation*}
V\colon \mathcal{F}(I)/\mathbb{R} \rightarrow T_p\mathcal{P}_+(I),\quad\quad 
V_\Phi=-\textrm{div}_G(p\nabla_G\Phi).
\end{equation*} 

In~\cite{chow2012} it is shown that $V_\Phi\colon  \mathcal{F}(I)/\mathbb{R}\rightarrow T_p\mathcal{P}_+(I)$ is a well defined map which is linear and  one-to-one. 
I.e., $ \mathcal{F}(I)/\mathbb{R}\cong T_p^*\mathcal{P}_+(I)$, where $T_p^*\mathcal{P}_+(I)$ is the cotangent space of $\mathcal{P}_+(I)$.
This identification induces the following inner product on 
$T_p\mathcal{P}_+(I)$. 

We first present this in a \emph{dual} formulation, which is known in the literature~\cite{Lott}. 
\begin{definition}[Inner product in dual coordinates]\label{d9}
The inner product 
$g_p^W :T_p\mathcal{P}_+(I)\times T_p\mathcal{P}_+(I)  \rightarrow \mathbb{R}$ 
takes any two tangent vectors $V_{\Phi}$ and $V_{\tilde\Phi}\in T_p\mathcal{P}_+(I)$ to 
\begin{equation}\begin{split}\label{formula}
g_p^W(V_{\Phi}, V_{\tilde\Phi})=(\nabla_G\Phi, \nabla_G\tilde\Phi)_p.
\end{split} 
\end{equation}
\end{definition}
We shall now give the inner product in \emph{primal} coordinates. 
The following matrix operator will be the key to the Riemannian metric tensor of $(\mathcal{P}_+(I), g^W)$. 

\begin{definition}[Linear weighted Laplacian matrix]\label{def2}
	Given $I=\{1,\ldots, n\}$ and a weighted graph $G=(I,E,\omega)$, the matrix function $L(\cdot):\mathbb{R}^n\rightarrow \mathbb{R}^{n\times n}$ is defined by
\begin{equation*}
L(a)=D^{\ts}\Lambda(a)D,\quad a=(a_i)_{i=1}^n\in \mathbb{R}^n,
\end{equation*}
where 
\begin{itemize}
\item 
$D \in \mathbb{R}^{|E|\times n}$ is the discrete gradient operator  
\begin{equation*} 
D_{(i,j)\in {E}, k\in V}=\begin{cases}
\sqrt{\omega_{ij}}, & \textrm{if $i=k$, $i>j$}\\ 
-\sqrt{\omega_{ij}}, & \textrm{if $j=k$, $i>j$}\\
0, & \textrm{otherwise}
\end{cases},
\end{equation*}
\item $-D^{\ts}\in \mathbb{R}^{n\times |E|}$ is the discrete divergence operator, also called oriented incidence matrix \cite{Graph}, 
and 
\item $\Lambda(a)\in \mathbb{R}^{|E|\times |E|}$ is a weight matrix depending on $a$, 
\begin{equation*}
\Lambda(a)_{(i,j)\in E, (k,l)\in E}=\begin{cases}
\frac{a_i+a_j}{2} & \textrm{if $(i,j)=(k,l)\in E$}\\ 
0 & \textrm{otherwise}
\end{cases}.
\end{equation*}
\end{itemize}
\end{definition}

Consider some $p\in\mathcal{P}_+(I)$. From spectral graph theory \cite{Graph}, we know that $L(p)$ can be decomposed as 
\begin{equation*}
L(p)=U(p)\begin{pmatrix}
0 & & &\\
& {\lambda_{1}(p)}& &\\
& & \ddots & \\
& & & {\lambda_{n-1}(p)}
\end{pmatrix}U(p)^{\ts} .
\end{equation*}
Here $0<\lambda_1(p)\leq\cdots\leq \lambda_{n-1}(p)$ are the eigenvalues of $L(p)$ in ascending order, 
and $U(p)=(u_0(p),u_1(p),\cdots, u_{n-1}(p))$ is the corresponding 
orthogonal matrix of eigenvectors with 
$$u_0=\frac{1}{\sqrt{n}}(1,\ldots, 1)^{\ts}.$$
We write $L(p)^{\dagger}$ for the pseudo-inverse of $L(p)$, i.e., 
\begin{equation*}
L(p)^{\dagger}=U(p)\begin{pmatrix}
0 & & &\\
& \frac{1}{\lambda_{1}(p)}& &\\
& & \ddots & \\
& & & \frac{1}{\lambda_{n-1}(p)}
\end{pmatrix}U(p)^{\ts} .
\end{equation*}
With $\sigma=L(p)\Phi$, $\tilde\sigma=L(p)\tilde\Phi$, we see that 
\begin{equation*}\label{relation}
{\sigma}^{\ts}L(p)^{\dagger}\tilde\sigma=\Phi^{\ts}L(p)L(p)^{\dd}L(p)\tilde\Phi=\Phi^{\ts}L(p)\tilde\Phi=(\nabla_G\Phi, \nabla_G\tilde\Phi)_p.
\end{equation*}

Now we are ready to give the inner product in primal coordinates. 
\begin{definition}[Inner product in primal coordinates]
	\label{definition:innerproduct}
The inner product $g^{W}_p:T_p\mathcal{P}_+(I)\times T_p\mathcal{P}_+(I)\rightarrow\mathbb{R}$ is defined by
\begin{equation*}
g^{W}_p(\sigma,\tilde\sigma):={\sigma}^{\ts}L(p)^{\dagger}\tilde\sigma,\quad \textrm{for any $\sigma,\tilde\sigma\in T_p\mathcal{P}_+(I)$}.
\end{equation*}
\end{definition}
In other words, the variational problem from Definition~\ref{def_metric} is a minimization of geometry energy functional in $\mathcal{P}_+(I)$, i.e., 
\begin{equation*}
W( p^0, p^1)^2=\inf_{p(t)\in \mathcal{P}_+(I),t\in[0,1]}\Big\{\int_0^1\dot p(t)^{\ts}L(p(t))^{\dd}\dot p(t)dt~\colon~ p(0)= p^0,~ p(1)= p^1\Big\}.
\end{equation*}
This defines a Wasserstein Riemannian structure on the probability simplex. For more details of Riemannian formulas see \cite{LiG}. Following~\cite{Lafferty} we could call $(\mathcal{P}_+(I), g^W)$ discrete density manifold. However, this could be easily confused with other notions from information geometry, and hence we will use the more explicit terminology {\em Wasserstein statistical manifold}, or Wasserstein manifold for short. 

\section{Wasserstein statistical manifold}
\label{sec3}

In this section we study parametric probability models endowed with the $L^2$-Wasserstein Riemannian metric. 
We define this in the natural way, by pulling back the Riemannian structure from the Wasserstein manifold that we discussed in the previous section. This allows us to introduce a natural gradient flow on the parameter space of a statistical model.  

Consider a statistical model defined by a triplet $(\Theta, I, p)$. 
Here, $I=\{1,\ldots, n\}$ is the sample space, 
$\Theta$ is the parameter space, which is an open subset of $\mathbb{R}^d$, $d\leq n-1$, 
and $p\colon \Theta\rightarrow \mathcal{P}_+(I)$ is the parametrization function, 
\begin{equation*}
 p(\theta)=(p_i(\theta))_{i=1}^n,\quad \theta\in \Theta. 
\end{equation*}
In the sequel we will assume that $\textrm{rank}(J_{\theta} p(\theta))=d$, so that the parametrization is locally injective. 

We define a Riemannian metric $g$ on $\Theta$ as the pull-back of metric $g^W$ on $\mathcal{P}_+(I)$. 
In other words, we require that $p\colon (\Theta,g)\rightarrow (\mathcal{P}_+(I), g^W)$ is an isometric embedding: 
\begin{equation*}
\begin{split}
g_\theta(a,b):=&g^W_{ p(\theta)}(d p(\theta)(a), d p(\theta)(b))\\
=&\big(d p(\theta)(a)\big)^{\ts}L( p(\theta))^{\dd}\big(d p(\theta)(b)\big).
\end{split}
\end{equation*}
Here $d p(\theta)(a)=\big(\sum_{j=1}^n\frac{\partial p_i(\theta)}{\partial\theta_j}a_j\big)_{i=1}^n=J_\theta p(\theta)a$, where $J_\theta p(\theta)$ is the Jacobi matrix of $ p(\theta)$ with respect to $\theta$. We arrive at the following definition.
\begin{definition}
For any pair of tangent vectors $a,b\in T_\theta \Theta=\mathbb{R}^d$, define 
\begin{equation*}
g_\theta(a,b):=a^{\ts}J_\theta  p(\theta)^{\ts}L( p(\theta))^{\dd}J_\theta  p(\theta)b,
\end{equation*}
where $J_\theta p(\theta)=(\frac{\partial p_i(\theta)}{\partial \theta_j})_{1\leq i\leq n, 1\leq j\leq d}\in \mathbb{R}^{n\times d}$ is the Jacobi matrix of the parametrization $p$, and $L( p(\theta))^{\dd}\in \mathbb{R}^{n\times n}$ is the pseudo-inverse of the linear weighted Laplacian matrix. 
\end{definition}

This inner product is consistent with the restriction of the Wasserstein metric $g^W$ to $ p(\Theta)$. 
For this reason, we call $ p(\Theta)$, or $(\Theta, I, p)$, together with the induced Riemannian metric $g$, \emph{Wasserstein statistical manifold}. 

We need to make sure that the embedding procedure is valid, because the metric tensor $L(p)^{\dd}$ is only of rank $n-1$. The next lemma shows that $(\Theta, g)$ is a well defined $d$-dimensional Riemannian manifold. 
\begin{lemma}
For any $\theta\in \Theta$, we have
\begin{equation*}
\lambda_{\min}(\theta)=\inf_{a\in\mathbb{R}^d, \|a\|_2=1}g_\theta(a, a)>0.
\end{equation*}
In addition, $g_{\theta}$ is smooth as a function of $\theta$, so that $(\Theta, g)$ is a smooth Riemannian manifold. 
\end{lemma}
\begin{proof}
We only need to show that $J_\theta  p(\theta)^{\ts}L( p(\theta))^{\dd}J_\theta  p(\theta)\in \mathbb{R}^{d\times d}$ is a positive definite matrix.
Consider 
\begin{equation*}
a^{\ts}J_\theta  p(\theta)^{\ts}L( p(\theta))^{\dd}J_\theta  p(\theta)a=0,
\end{equation*}
where $0\in \mathbb{R}^{n-1}$. Since $L(p)$ only has one simple eigenvalue $0$ with eigenvector $u_0$, then 
\begin{equation}\label{a}
J_\theta  p(\theta)a=cu_0,\quad \textrm{for some constant $c\in \mathbb{R}^1$.}
\end{equation}
Since $u_0^{\ts} p(\theta)=\frac{1}{\sqrt{n}}\sum_{i=1}^np_i(\theta)=0$, we have that  $u_0^{\ts}\frac{\partial p(\theta)}{\partial\theta_j}=\frac{1}{\sqrt{n}}\sum_{i=1}^n\frac{\partial p_i(\theta)}{\partial\theta_j}=0$, i.e., 
\begin{equation*}
u_0^{\ts}J_\theta  p(\theta)=0.
\end{equation*}
Left multiply $u_0$ into \eqref{a}, we obtain  
\begin{equation*}
0=u_0^{\ts}J_\theta  p(\theta)a=cu_0^{\ts}u_0=c.
\end{equation*}
Thus $c=0$, and \eqref{a} forms 
\begin{equation*}
J_\theta  p(\theta)a=0.
\end{equation*}
Since $\textrm{rank}(J_{\theta} p(\theta))=d<n$, we have $a=0$, which finishes the proof.
\end{proof}

We illustrate some geometric calculations on parameter space $(\Theta, g)$. For simplicity of illustration, we assume $\Theta\subset\mathbb{R}^d$, and denote a matrix function $G(\theta)\in \mathbb{R}^{d\times d}$ with $g_\theta(\dot\theta, \dot\theta)=\dot\theta^{\ts}G(\theta)\dot\theta$, i.e.,
\begin{equation}\label{metric_tensor}
G(\theta)=(J_\theta p(\theta))^{\ts}L(p(\theta))^{\dd}(J_\theta p(\theta)).
\end{equation}
Under this notation, given $\theta_0$, $\theta_1\in \Theta$, the Riemannian distance on $(\Theta, g)$ is defined by the geometric action functional:
\begin{equation}\label{Dist}
\textrm{Dist}(\theta_0,\theta_1)^2=\inf_{\theta(\cdot)\in {C^1([0,1];\Theta)}}\Big\{\int_0^1{\dot\theta(t)^{\ts}G(\theta(t))\dot\theta(t)}dt~:~\theta(0)=\theta_0,~\theta(1)=\theta_1\Big\}.
\end{equation}
Denote $\theta(t)=\theta_t$, and $S_t$ is the Legendre transformation of $\dot\theta_t$ in $(\Theta, g)$, then the cotangent geodesic flow satisfies
\begin{equation}\label{cotangent}
\begin{cases}
\dot \theta_t-G(\theta_t)^{-1}S_t=0\\ 
\dot S_t+\frac{1}{2}\frac{\partial}{\partial\theta} S^{\ts}_tG(\theta_t)^{-1}S_t=0.
\end{cases}
\end{equation}

It is worth recalling the following facts. If $p$ is an identity map, then \eqref{cotangent} translates to 
\begin{equation*}
\begin{cases}
\dot p+\textrm{div}_G(p\nabla_G S)=0\\ 
\dot S+\frac{1}{4}\sum_{j\in N(i)}(\nabla_GS)^2=0.
\end{cases}
\end{equation*}
In addition, if $I=\Omega$ and we replace $i$ by $x$ and $p_i(t)$ by $\rho(t,x)$, the above becomes 
\begin{equation*}
\begin{cases}
\frac{\partial \rho(t,x)}{\partial t}+\textrm{div}(\rho(t,x)\nabla S(t,x))=0\\ 
\frac{\partial S(t,x)}{\partial t}+\frac{1}{2}(\nabla S(t,x))^2=0,
\end{cases}
\end{equation*}
which are the standard continuity and Hamilton-Jacobi equations on $\Omega$. 
For these reasons, we call the two equations in \eqref{cotangent}  the {\em continuity equation} and the {\em Hamilton-Jacobi equation} {\em on parameter space}. 

\section{Gradient flow on Wasserstein statistical manifold}\label{secnew}
In this section we introduce the natural Riemannian gradient flow on Wasserstein statistical manifold $(\Theta, g)$. 

\subsection{Gradient flow on parameter space}
Consider a smooth loss function $F\colon \mathcal{P}_+(I)\rightarrow \mathbb{R}$.  Thus we focus on the composition 
$F\circ p\colon \Theta\rightarrow \mathbb{R}$. 
The Riemannian gradient of $F( p(\theta))$ is defined as follows. Given $\nabla_g F( p(\theta))\in T_\theta \Theta$, we have 
\begin{equation}\label{gfdd}
g_\theta(\nabla_g F( p(\theta)), a)=
\nabla_\theta{ F}( p(\theta))\cdot a,\quad 
\textrm{for any}~ a\in T_\theta \Theta,
\end{equation}
where $\nabla_\theta  F( p(\theta))\cdot a=
\sum_{i=1}^d\frac{\partial}{\partial \theta_i}F( p(\theta))a_i$. 
The gradient flow satisfies 
\begin{equation*}
\dot\theta_t=-\nabla_g F( p(\theta_t)).
\end{equation*}

The next theorem establishes an explicit formulation of the gradient flow. 
\begin{theorem}[Wasserstein gradient flow] 
	\label{thetheorem}
The gradient flow of a functional $F\colon \mathcal{P}_+(I)\rightarrow \mathbb{R}$ is given by 
\begin{equation*}
\dot\theta_t=-G(\theta_t)^{-1}\nabla_\theta  F( p(\theta_t)),
\end{equation*}
where $\nabla_{\theta}$ is the Euclidean gradient of $F( p(\theta))$ with respect to~$\theta$. 
More explicitly, 
\begin{equation}\label{a1}
\dot\theta_t=-\Big(J_\theta  p(\theta_t)^{\ts}L( p(\theta_t))^{\dd}J_\theta  p(\theta_t)\Big)^{\dd}J_\theta  p(\theta_t)^{\ts}\nabla_ pF( p(\theta_t),
\end{equation}
where $\nabla_p$ is the Euclidean gradient of $F(p)$ with respect to~$p$. 
\end{theorem}
\begin{proof}
The proof follows directly from \eqref{gfdd}. Notice that 
\begin{equation*}
g_\theta(\nabla_gF( p(\theta)), a)=\nabla_g F( p(\theta))^{\ts}J_\theta  p(\theta)^{\ts}L( p(\theta))^{\dd}J_\theta  p(\theta) a=\nabla_\theta F( p(\theta))^{\ts}a,
\end{equation*}
and $J_\theta  p(\theta)^{\ts}L( p(\theta))^{\dd}J_\theta  p(\theta)$ is an invertible matrix. 
Hence 
\begin{equation*}
\nabla_gF( p(\theta))=\big(J_\theta  p(\theta)^{\ts}L( p(\theta))^{\dd}J_\theta  p(\theta)\big)^{\dd}\nabla_\theta F( p(\theta)).
\end{equation*}
We compute $\nabla_\theta F( p(\theta))$ as
\begin{equation*}
\nabla_\theta F( p(\theta))=\big(\frac{\partial}{\partial\theta_i} F( p(\theta))
\big)_{i=1}^n= \big(\sum_{j=1}^n\frac{\partial}{\partial p_j} F( p(\theta))\cdot\frac{\partial p_j(\theta)}{\partial\theta_i} \big)_{i=1}^n =J_\theta  p(\theta)^{\ts}\nabla_ pF( p(\theta)).
\end{equation*}
This concludes the proof of \eqref{a1}.
\end{proof}

{Equation \eqref{a1} is the generalization of  Wasserstein gradient flow in probability simplex to the one on parameter space. If $p$ is an identity map with the parameter space $\Theta$ equal to the entire probability simplex, then \eqref{a1} is 
\begin{equation*}
\dot p_t=-\nabla_g F(p_t)=\textrm{div}_G(p_t\nabla_G\nabla_p F(p_t)),
\end{equation*}
which is the Wasserstein gradient flow {on the discrete probability simplex}.} In particular, if $I=\Omega$, then it represents 
\begin{equation*}
\partial_t\rho_t=-\nabla_W F(\rho_t)=\operatorname{div}(\rho_t \nabla  \delta_\rho  F(\rho_t)),
\end{equation*}
which is the Wasserstein gradient flow on $\Omega$. From now on, we call \eqref{a1} the {\em Wasserstein gradient flow on parameter space}.

The definition of Wasserstein gradient flow shares many similarities with the {steepest gradient descent defined as follows}. 
Consider  
\begin{equation}\label{min}
\arg\min_{ h\in T_\theta\Theta} F( p(\theta+ h)) \quad \textrm{s.t.}\quad  \frac{1}{2}W( p(\theta),  p(\theta+ h))^2=\epsilon,
\end{equation}
where $\epsilon\in \mathbb{R}_+$ is a given small constant. By taking the second-order Taylor approximation of the Wasserstein distance at $\theta$, we get 
\begin{equation*}
W( p(\theta),  p(\theta+ h))^2= h^{\ts}G(\theta) h+o( h^2),
\end{equation*}
where $G(\theta)$ is the metric tensor of $(\Theta, g)$ defined in \eqref{metric_tensor}, inherited from Wasserstein manifold. 
We take the first-order approximation of $F( p(\theta+ h))$  in \eqref{min} by 
\begin{equation*}
\arg\min_{ h\in T_\theta\Theta} F( p(\theta)) + h^{\ts}\nabla_\theta F( p(\theta)) \quad \textrm{s.t.}\quad \frac{1}{2} h^{\ts}G(\theta) h=\epsilon.
\end{equation*}
By the Lagrangian method with Lagrange multiplier $\lambda$, we have
\begin{equation*}
 h=\lambda G(\theta)^{-1}\nabla_{\theta}F( p(\theta)).
\end{equation*}
The above derivations lead to the Wasserstein natural gradient direction $$\nabla_g F( p(\theta))=G(\theta)^{-1}\nabla_\theta F( p(\theta)).$$ 

\begin{remark}
In the standard Fisher-Rao natural gradient~\cite{NG}, we replace \eqref{min} by  
\begin{equation*}
\arg\min_{ h} F( p(\theta+ h)) \quad \textrm{s.t.}\quad \textrm{KL}( p(\theta) \|  p(\theta+ h))=\epsilon,
\end{equation*}
where $\operatorname{KL}$ stands for the Kullback-Leibler divergence (relative entropy) from $ p(\theta)$ to $ p(\theta+ h)$. 
Our definition changes the KL-divergence by the Wasserstein distance. 
\end{remark}

\subsection{Displacement convexity on parameter space}
The Wasserstein structure on the statistical manifold not only provides us the gradient operator, but also the Hessian operator on $(\Theta, g)$. The latter allows us to introduce the displacement convexity on parameter space. 

We first review some facts. 
One remarkable property of Wasserstein geometry is that it yields a correspondence between differential operators on sample space and differential operators on probability space. 
E.g., the Hessian operator on Wasserstein manifold is equal to the expectation of Hessian operator on sample space. 
 
 An important example is stochastic relaxation. Given $f(x)\in C^{\infty}(\Omega)$, consider
\begin{equation*} 
 F(\rho)=\mathbb{E}_{X\sim\rho}[f(X)]=\int_\Omega f(x)\rho(x)dx.
\end{equation*}
It is known that the Hessian operator of $ F(\rho)$ on Wasserstein manifold satisfies 
\begin{equation*}\label{HessF}
\begin{split}
\operatorname{Hess}_W F(\rho)(V_{\Phi}, V_{\tilde\Phi})=\mathbb{E}_{X\sim\rho}(\operatorname{Hess}f(X) \nabla\Phi(X), \nabla\tilde{\Phi}(X)).
\end{split}
\end{equation*}
One can 
show that  $\operatorname{Hess}f\succeq \lambda \mathbb{I}$ if and only if $\operatorname{Hess}_W F(\rho)(V_\Phi, V_\Phi)\succeq \lambda g^W_\rho(V_\Phi, V_\Phi)$. 
This means that $f$ is $\lambda$-geodesic convex in $(\Omega, g^\Omega)$ if and only if $ F(\rho)$ is $\lambda$-geodesic convex in $(\mathcal{P}(\Omega), g^W)$. In literature \cite{vil2008}, the geodesic convexity on Wasserstein manifold is known as the displacement convexity. 

In this section we would like to extend the displacement convexity to parameter space $\Theta$. In other words, we relate the parameter to the differential structures of sample space via constrained Wasserstein geometry $(\Theta, g)$. 
If $\Theta$ is the full probability manifold, our definition coincides with the classical Hessian operator in sample space. 

\begin{definition}[Displacement convexity on parameter space]
Given $ F\circ p\colon\Theta \rightarrow \mathbb{R}$, we say that $F(p(\theta))$ is $\lambda$-displacement convex if
for any constant speed geodesic $\theta_t$, $t\in[0,1]$ connecting $\theta_0,\theta_1\in (\Theta, g)$, it holds that
\begin{equation*}
F(p(\theta_t))\geq (1-t) F(p(\theta_0))+t F(p(\theta_1))-\frac{\lambda}{2}t(1-t)\operatorname{Dist}(\theta_0,\theta_1)^2,
\end{equation*}
where $\operatorname{Dist}$ is defined in \eqref{Dist}. If $F(p(\theta))=\sum_{i=1}^n f_ip_i(\theta)$ is $\lambda$-displacement convex, then we call $f\in \mathbb{R}^n$ $\lambda$-convex in $(\Theta, I, p)$.
\end{definition}
\begin{remark}
In particular, the displacement convexity of KL divergence relates to the Ricci curvature lower bound on sample space. We elaborate this notion in \cite{RLG}. 
\end{remark}
We next derive the displacement convexity condition for stochastic relaxation. 
\begin{theorem}
Assume $\Theta\subset \mathbb{R}^d$ is a compact set and $f=(f_i)_{i=1}^n\in \mathbb{R}^n$. Then $f$ is $\lambda$-convex if and only if
\begin{equation}\label{DC}
\sum_{i=1}^n p_i(\theta) \Big(\Gamma(\Gamma(f, \Phi),\Phi)-\frac{1}{2}\Gamma(\Gamma(\Phi, \Phi), f)\Big)_i+\sum_{i=1}^nf_i B_{p(\theta)}(V_\Phi, V_\Phi)_i\geq \lambda \sum_{i=1}^n \Gamma(\Phi, \Phi)_i p_i(\theta),
\end{equation}
for any $\Phi\in \mathcal{F}(I)/\mathbb{R}$ and $\theta\in \Theta$. Here $\Gamma\colon \mathbb{R}^n\times \mathbb{R}^n \rightarrow \mathbb{R}^n$ is given by 
\begin{equation*}
\Gamma(\Phi,\tilde\Phi)_i\colon=g_i^I(\nabla_G\Phi, \nabla_G\tilde\Phi)=\frac{1}{2}\sum_{j\in N(i)}\omega_{ij}(\Phi_{i}-\Phi_{j})(\tilde\Phi_{i}-\tilde\Phi_{j}),
\end{equation*}
and $B$ is the second fundamental form given in Proposition~\ref{prop8}.
\end{theorem}
\begin{proof}
If $\Theta$ is a compact set, then the $\lambda$-displacement convexity of $F(p(\theta))$ is equivalent to 
\begin{equation*}
\operatorname{Hess}_g F(p(\theta))\succeq\lambda G(\theta),
\end{equation*}
where $\operatorname{Hess}_g F$ is the Hessian operator in $(\Theta, g)$. We next calculate this Hessian operator explicitly. 
Notice that
\begin{equation*}
\operatorname{Hess}_gF(\sigma, \tilde\sigma)=\operatorname{Hess}_WF(\sigma,\tilde\sigma)+B_{p(\theta)}(\sigma, \tilde\sigma)^{\ts}\nabla_p F( p(\theta)),
\end{equation*}
where $\textrm{Hess}_W$ is the Hessian operator in $(\mathcal{P}_+(I), g^W)$.
Denote the above in dual coordinates, i.e. $\sigma=\tilde\sigma=V_\Phi=V_{\tilde\Phi}=L(p(\theta))\Phi$, and follow the geometric computations in \cite[Proposition~18]{LiG}, we finish the proof. 
\end{proof}
Here $\Gamma$ is the discrete Bakry-Emery Gamma one operator~\cite{BE}. The geometry of Wasserstein manifold is directly related to the expectation of Bakry-Emery Gamma one operators \cite{LiG}. 
In particular, if $p$ is the identity mapping and $I=\Omega$, then our definition \eqref{DC} represents 
\begin{equation*}
\int_\Omega \Big(\Gamma(\Gamma(f, \Phi),\Phi)-\frac{1}{2}\Gamma(\Gamma(\Phi, \Phi), f)\Big)\rho(x)dx\geq \lambda \Gamma(\Phi,\Phi)\rho dx,
\end{equation*}
i.e. 
\begin{equation*}
\int_\Omega \operatorname{Hess}f(x) (\nabla\Phi(x),\nabla\Phi(x))\rho(x)dx\geq  \lambda \int_\Omega g^\Omega_x(\nabla\Phi, \nabla\Phi)\rho(x)dx
\end{equation*}
for any $\rho$, and vector fields $\nabla\Phi$. The above inequality is same as requiring $\textrm{Hess}f\succeq \lambda I$. Our definition extends this concept to parameter space.

\subsection{Numerical methods}
Here we discuss the numerical computation of the Wasserstein metric and the Wasserstein gradient flow. 

Given the gradient flow~\eqref{a}, there are two standard choices of time discretization, namely the forward Euler scheme and the backward Euler scheme. 
Denote the step size by $\lambda>0$. 
The forward Euler method 
computes a discretized trajectory by 
\begin{equation*}
\theta^{k+1}=\theta^k-\lambda \nabla_g  F( p(\theta^k)),
\end{equation*}
while the backward Euler method computes 
\begin{equation*}
\theta^{k+1}=\arg\min_{\theta\in \Theta} F( p(\theta))+\frac{\textrm{Dist}(\theta, \theta^k)^2}{2\lambda},
\end{equation*}
where $\textrm{Dist}$ is the geodesic distance in parameter space $(\Theta, g)$. 

In the information geometry literature, the forward Euler method is often referred to as natural gradient method. 
In Wasserstein geometry, the backward Euler method is often called the Jordan-Kinderlehrer-Otto (JKO) scheme. In the following we give pseudo code for both numerical methods. 

\begin{table}[H]
\begin{tabbing}
aaaaa\= aaa \=aaa\=aaa\=aaa\=aaa=aaa\kill  
   \rule{\linewidth}{0.8pt}\\
   \noindent{\large\bf Natural Wasserstein gradient method}\\
    \rule{\linewidth}{0.8pt}\\ 
  \1 \For $k=1, 2, \ldots$ \quad \textrm{while not converged}\\
1. \3 Choose a suitable step size $\lambda_k>0$\ ; \\ 
2. \3 $\theta^{k+1} = \theta^k-\lambda_k\big((J_{\theta}  p(\theta^k))^{\ts}L( p(\theta^k))^{\dd}J_\theta  p(\theta^k)\big)^{\dd}(J_\theta  p(\theta^k))^{\ts}\nabla_ pF( p(\theta^k))$\ ; \\
  \1 \End\\
   \rule{\linewidth}{0.8pt}
\end{tabbing}
\label{naturalWassersteinpage}
\end{table}

\begin{table}[H]
\begin{tabbing}
aaaaa\= aaa \=aaa\=aaa\=aaa\=aaa=aaa\kill  
   \rule{\linewidth}{0.8pt}\\
   \noindent{\large\bf Natural Jordan-Kinderlehrer-Otto scheme}\\
    \rule{\linewidth}{0.8pt}\\ 
  \1 \For $k=1, 2, \ldots$ \quad \textrm{while not converged}\\
1. \3 Choose a suitable adaptive step size $\lambda_k>0$\ ;\\
2. \3 $\theta^{k+1} = \arg\min_{\theta\in \Theta} F( p(\theta))+\frac{\textrm{Dist}(\theta, \theta^k)^2}{2\lambda_k} $\ ; \\
  \1 \End\\
   \rule{\linewidth}{0.8pt}
\end{tabbing}
\end{table}
In practice, the forward Euler method is usually easier to implement than the backward Euler method. We would also suggest to implement the natural Wasserstein gradient using this method for minimization problems. 
As known in optimization, the JKO scheme can also be useful for non-smooth objective functions. Moreover, the backward Euler method is usually unconditionally stable, which means that one can choose a large step size $h$ for computations.

\section{Examples}
\label{sec4}
\begin{example}[Wasserstein geodesics]
	\label{example:paths}
Consider the sample space $I=\{1,2,3\}$ with an unweighted graph $1-2-3$. 
The probability simplex for this sample space is a triangle in~$\mathbb{R}^3$: 
\begin{equation*}
\mathcal{P}(I)=\Big\{(p_i)_{i=1}^3\in \mathbb{R}^3~:~\sum_{i=1}^3p_i=1, \quad p_i\geq 0\Big\}. 
\end{equation*}
Following Definition~\ref{def_metric}, the $L^2$-Wasserstein distance is given by 
\begin{equation}\label{metric}
W(p^0, p^1)^2:=\inf_{\Phi(t)}\int_0^1 \Big\{(\Phi_1(t)-\Phi_2(t))^2\frac{p_1(t)+p_2(t)}{2}+(\Phi_2(t)-\Phi_3(t))^2\frac{p_2(t)+p_3(t)}{2}\Big\} dt, 
\end{equation}
where the infimum is taken over paths $\Phi\colon[0,1]\rightarrow \mathbb{R}^3$. 
Each $\Phi$ defines $p\colon[0,1]\rightarrow \mathbb{R}^3$ as the solution of the differential equation 
\begin{equation*}
\begin{cases}
\dot p_1=&(\Phi_1-\Phi_2)\frac{p_1+p_2}{2}\\
\dot p_2=&(\Phi_2-\Phi_1)\frac{p_1+p_2}{2}+(\Phi_2-\Phi_3)\frac{p_2+p_3}{2}\\
\dot p_3=&(\Phi_3-\Phi_2)\frac{p_2+p_3}{2}
\end{cases} 
\end{equation*}
with boundary condition $p(0)=p^0$, $p(1)=p^1$. 

Consider local coordinates in~\eqref{metric}. We parametrize a probability vector as $p=(p_1, 1-p_1-p_3, p_3)$, with parameters $(p_1, p_3)$. 
Then \eqref{metric} can be written as 
\begin{equation}\label{WM}
W(p^0, p^1)^2:=\inf_{p(t)\colon p(0)=p^0,~p(1)=p^1}\int_0^1\Big\{\frac{\dot p_1(t)^2}{1-p_3(t)}+\frac{\dot p_3(t)^2}{1-p_1(t)}\Big\}dt.
\end{equation}
where the infimum is taken over paths $p\colon[0,1]\rightarrow \mathcal{P}_+(I)$. We also compare the Wasserstein metric \eqref{WM} with the Fisher-Rao metric. In this case, the Fisher-Rao metric function is given by
\begin{equation*}
\textrm{FR}(p^0, p^1)^2:=\inf_{p(t)\colon p(0)=p^0,~p(1)=p^1}\int_0^1\Big\{\frac{\dot p_1(t)^2}{p_1(t)}+\frac{(\dot p_1(t)+\dot p_3(t))^2}{p_2(t)}+\frac{\dot p_3(t)^2}{p_3(t)}\Big\}dt.
\end{equation*}
This clearly demonstrates the difference between Wasserstein Riemannian metric and Fisher-Rao metric. We would also compare the dynamical optimal transport with the statistical one. In particular, if the ground metric is given by $c_{12}=1$, $c_{13}=2$, $c_{23}=1$, 
which is of homogenous degree one type. Then the statistical optimal transport defined by 
\begin{equation*}
d(p^0, p^1)=\inf_{\pi\geq 0}\Big\{c_{12}\pi_{12}+c_{13}\pi_{13}+c_{12}\pi_{23}\colon \sum_{i=1}^3\pi_{ij}=p^0_j,~\sum_{j=1}^3\pi_{ij}=p^1_i\Big\},
\end{equation*}
can be reformulated by 
\begin{equation*}
{d}(p^0, p^1)=\inf_{p(t)\colon p(0)=p^0,~p(1)=p^1}\int_0^1\Big\{|\dot p_1(t)|+|\dot p_3(t)|\Big\}dt.
\end{equation*}
Here the statistical formulation does not provide a Riemannian metric, but gives a Finslerian metric. 

We next compute  \eqref{WM} numerically\footnotemark for different choices of the boundary conditions $p^0$, $p^1$. \footnotetext{We use the \emph{direct method}, which is a standard technique in optimal control. Here the time is discretized, and the sum replacing the integral is minimized by means of gradient descent with respect to $(p(t)_i)_{i=1,3, t\in\{t_1,\ldots, t_N\}} \in \mathbb{R}^{2\times N}$. A 
	reference for these techniques is~\cite{Li_COM}.} We fix three distributions 
\begin{equation}
q^1=\frac18(6, 1, 1), \quad 
q^2=\frac18(1, 6, 1),\quad 
q^3=\frac18(1, 1, 6) 
\end{equation} 
and solve \eqref{WM} for three choices of the boundary conditions:  
\begin{equation}
 p^0=q^1,\; p^1=q^2 ; \quad 
 p^0=q^1,\; p^1=q^3 ; \quad 
p^0=q^2,\; p^1=q^3  . 
\end{equation} 
This gives us a geodesic triangle between $q^1, q^2, q^3$, which is illustrated in Figure~\ref{fig1}. 
It can be seen that $(\mathcal{P}_+(I), W)$ has a non Euclidean geometry. 
Moreover, we see that the geodesics depend on the graph structure on sample space, where state $2$ is qualitatively different from states $1$ and~$3$. 

\begin{figure}
	\begin{center}
		\begin{tikzpicture}[->,shorten >=1pt,auto,node distance=2cm,
		thick,main node/.style={circle,fill=blue!20,draw,minimum size=0.5cm,inner sep=0pt]} ]
		\node[main node] (1) {$1$};
		\node[main node] (2) [right of=1]  {$2$};
		\node[main node](3)[right of=2]{$3$};
		\path[-]
		(1) edge node {} (2)
		(2) edge node{} (3);
		\end{tikzpicture}
	\end{center}
	\centering
	\begin{tikzpicture}
	\node[] (A) at(0,0) {\includegraphics[clip=true, trim=2cm 10.5cm 10.5cm 10.5cm, width=0.4\textwidth]{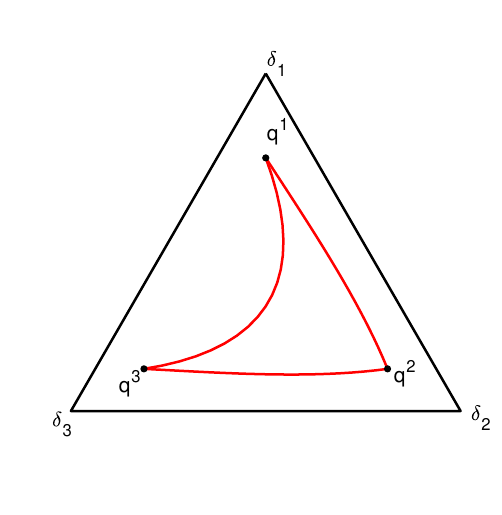}}; 
	\node[] (B) at (7,0) {\includegraphics[clip=true, trim=10.5cm 10.5cm 2cm 10.5cm, width=0.4\textwidth]{wpaths}};
	\draw[->] (2,1) to [bend right]node[below] {$\theta$}  (4,1) ;
	\draw[<-] (2,1.5) to [bend left]node[above] {$p$}  (4,1.5) ;
	\end{tikzpicture}
	\caption{The Wasserstein geodesic triangle from Example~\ref{example:paths} plotted in the probability simplex (left) and in exponential parameter space (right). 
		The path connecting $q^1$ and $q^3$ bends towards~$q^2$; something that does not happen for the other two paths. 
		This illustrates how, as a result of the ground metric on sample space, state $2$ is treated differently from $1$ and $3$. 
	}
	\label{fig1}
\end{figure}

We can make the same derivations in terms of an exponential parametrization. 
Consider the parameter space $\Theta=\{\theta=(\theta_1,\theta_2)\in \mathbb{R}^2\}$ and the parametrization 
$p\colon \Theta\rightarrow \mathcal{P}_+(I)$ with 
\begin{equation*}
p_1(\theta)=\frac{e^{\theta_1}}{e^{\theta_1}+e^{\theta_2}+1}, \quad p_3(\theta)=\frac{e^{\theta_2}}{e^{\theta_1}+e^{\theta_2}+1},\quad p_2(\theta)=1-p_1(\theta)-p_3(\theta)=\frac{1}{e^{\theta_1}+e^{\theta_2}+1}.
\end{equation*}
We rewrite the Wasserstein metric \eqref{WM} in terms of $\theta$. 
Denote $ p(\theta^k)=p^k$, $k=0, 1$. 
Then the Wasserstein metric in the coordinate system $\theta$ is 
\begin{equation*}
\operatorname{Dist}(\theta^0,\theta^1)^2=\inf_{\theta(t)\colon \theta(0)=\theta^0,~\theta(1)=\theta^1}\Big\{\int_0^1\dot\theta^{\ts}J_\theta(p_1, p_3)^{\ts}\begin{pmatrix}
\frac{1}{1-p_3(\theta)} & 0\\
0 & \frac{1}{1-p_1(\theta)}
\end{pmatrix}  J_\theta(p_1, p_3) \dot\theta dt\Big\}.
\end{equation*}
The resulting geodesic triangle in $\Theta$ is plotted in the right panel of Figure~\ref{fig1}. 

For comparison, we compute the \emph{exponential} geodesic triangle between the same distributions $q^1,q^2,q^3$. This is shown in Figure~\ref{fig1a}. 
In this case, there is no distinction between the states $1,2,3$ and the three paths are symmetric. 
The exponential geodesic between two distributions $p^0$ and $p^1$ is given by $(p^0)^{1-t}(p^1)^{t}/ \sum_x(p^0)^{1-t}(p^1)^{t}$, $t\in[0,1]$. 

\begin{figure}
	\centering
\begin{tikzpicture}
\node[] (A) at(0,0) {\includegraphics[clip=true, trim=2cm 10.5cm 10.5cm 10.5cm, width=0.4\textwidth]{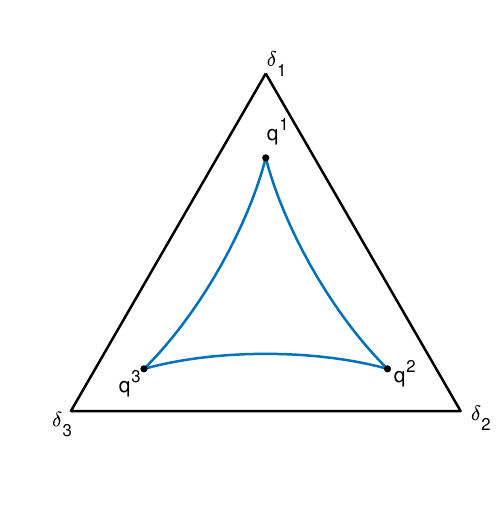}}; 
\node[] (B) at (7,0) {\includegraphics[clip=true, trim=10.5cm 10.5cm 2cm 10.5cm, width=0.4\textwidth]{epaths}};
\draw[->] (2,1) to [bend right]node[below] {$\theta$}  (4,1) ;
\draw[<-] (2,1.5) to [bend left]node[above] {$p$}  (4,1.5) ;
\end{tikzpicture}
	\caption{Exponential geodesic triangle plotted in the probability simplex (left) and in exponential parameter space (right). 
	Exponential geodesics correspond to straight lines in exponential parameter space. 
	}
	\label{fig1a}
\end{figure}

\end{example}

\begin{example}[Wasserstein gradient flow on an independence model]
	\label{example:independencemodel}
	
We next illustrate the Wasserstein gradient flow over the independence model of two binary variables. 
The sample space is $I = \{-1,+1\}^2$. 
For simplicity, we denote the states by $a=(-1,-1)$, $b=(-1, +1)$, $c=(+1,-1)$, $d=(+1,+1)$. 
We consider the square graph 
$$
\begin{matrix}
b -d\\
| \phantom{--}|\\
a- c
\end{matrix}
$$
with vertices $I$, edges $E=\{\{a, b\},\{b,d\}, \{a,c\},\{c,d\} \}$, 
and weights $\omega
=( \omega_{ab}, \omega_{bd}, \omega_{ac}, \omega_{cd})\in\mathbb{R}^E$ attached to the edges. 
The edge weights correspond to the inverse squared ground metric that we assign to  the sample space $I$. 
The probability simplex for this sample space is the tetrahedron 
\begin{equation*}
\mathcal{P}(I) = \Big\{(p(x))_{x\in I}\in \mathbb{R}^4~:~\sum_{x\in I}p(x)=1, \quad p(x)\geq 0\Big\}. 
\end{equation*}
Following Definition~\ref{definition:innerproduct}, the Wasserstein metric tensor is given by 
$g_p^W=L(p)^{\dd}$, 
which is the inverse of the linear weighted Laplacian metric $L$ from Definition~\ref{def2}. In this example the latter is 
\begin{equation*}
L(p)=
\small{  
\begin{pmatrix}
\omega_{ab}\frac{p_a+p_b}{2}+\omega_{ac}\frac{p_a+p_c}{2}& -\omega_{ab}\frac{p_a+p_b}{2}&-\omega_{ac}\frac{p_a+p_c}{2}& 0  \\
-\omega_{ab}\frac{p_a+p_b}{2}& \omega_{ab}\frac{p_a+p_b}{2}+\omega_{bd}\frac{p_b+p_d}{2}  & 0 &  -\omega_{bd}\frac{p_b+p_d}{2}\\
    -\omega_{ac}\frac{p_a+p_c}{2}        &    0    &\omega_{ac}\frac{p_a+p_c}{2}+\omega_{cd} \frac{p_c+p_d}{2}& -\omega_{cd}\frac{p_c+p_d}{2} \\
     0    &\omega_{bd}\frac{p_b+p_d}{2} & -\omega_{cd}\frac{p_c+p_d}{2} &  \omega_{bd}\frac{p_b+p_d}{2}+\omega_{cd}\frac{p_c+p_d}{2}   \\
\end{pmatrix}
}
.
\end{equation*}

The independence model consist of the joint distributions that satisfy $p(x_1, x_2)=p(x_1)p(x_2)$. 
This can be parametrized in terms of $\Theta=\{\xi=(\xi_1,\xi_2)\in [0,1]^2\}$, where
 $\xi_1=p_1(x_1=+1)$, $\xi_2=p_2(x_2=+1)$ describe the marginal probability distributions. 
The parametrization $p\colon \Theta\rightarrow \mathcal{P}(I)$ is then 
 \begin{equation*}
  p(\xi)(x_1,x_2)=\begin{cases}
(1-\xi_1)(1-\xi_2)  &\textrm{if $(x_1,x_2)=(-1, -1)$} \\
 (1-\xi_1)\xi_2 & \textrm{if $(x_1,x_2)=(-1, +1)$} \\
\xi_1(1-\xi_2)& \textrm{if $(x_1,x_2)=(+1, -1)$}  \\
 \xi_1\xi_2& \textrm{if $(x_1,x_2)=(+1, +1)$}  
 \end{cases}.
 \end{equation*}
The model $ p(\Theta) \subset \mathcal{P}(I)$ is a two dimensional manifold. 
The parameter space $\Theta$ inherits the Riemannian structure  $g^W$ from $\mathcal{P}(I)$, which is computed as follows. 
Denote the Jacobi matrix of the parametrization by
\begin{equation*}
J_\xi p(\xi)=\begin{pmatrix}
-(1-\xi_2)&-(1-\xi_1) \\
-\xi_2&1-\xi_1 \\ 
1-\xi_2& -\xi_1 \\
\xi_2&\xi_1
\end{pmatrix}\in \mathbb{R}^{4\times 2}.
\end{equation*}
Then $g^W$ induces a metric tensor on $\Theta$ given by 
$$G(\xi)=J_\xi p(\xi)^{\ts} L( p(\xi))^{\dd} J_\xi( p(\xi))\in \mathbb{R}^{2\times 2}.$$

We now consider a discrete optimization problem via stochastic relaxation and illustrate the gradient flow. 
Consider following potential  function on $I$, taken from~\cite{Malago2}: 
$$f(x_1,x_2)=x_1+2x_2+3x_1x_2=\begin{cases}
0 &\textrm{if $(x_1,x_2)=(-1, -1)$} \\
-2& \textrm{if $(x_1,x_2)=(-1, +1)$} \\
-4 & \textrm{if $(x_1,x_2)=(+1, -1)$}  \\
6 & \textrm{if $(x_1,x_2)=(+1, +1)$}  
 \end{cases}.
$$
We are to minimize $F(\p)=\mathbb{E}_\p[ f]$, i.e.,  
\begin{equation*}
F( p(\xi))=\sum_{(x_1,x_2)\in I}f(x_1,x_2)p_1(x_1)p_2(x_2)=-4\xi_1-2\xi_2+12\xi_1\xi_2.
\end{equation*}

By Theorem~\ref{thetheorem}, the Wasserstein gradient flow is 
\begin{equation*}
\dot\xi=-G(\xi)^{-1}\nabla_\xi F( p(\xi)). 
\end{equation*}
For our function, the standard Euclidean gradient is $\nabla_\xi F( p(\xi))=(-4+12\xi_2, -2+12\xi_1)^{\ts}$. 
The matrix $G$ is computed numerically from $J$ and $L$. 
 
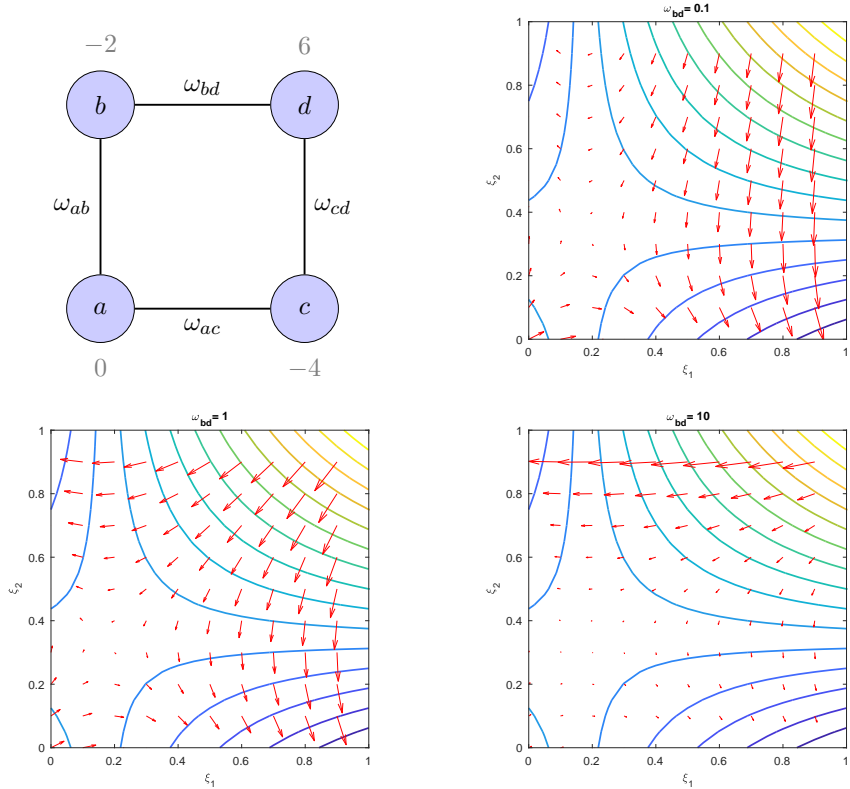
\begin{figure}[h]
\centering
\begin{tabular}{cc}
\scalebox{.9}{ 
	\begin{tikzpicture}[scale=0.5][-,shorten >=1pt,auto,node distance=2cm,
	thick,main node/.style={circle,fill=blue!20,draw,minimum size=0.1cm,inner sep=0pt]}]
	\node[main node] (1) {$a$};
	\node[main node] (2) [above =2cm of 1]  {$b$};
	\node[main node] (3) [right =2cm of 1]  {$c$};
	\node[main node] (4) [above =2cm of 3]  {$d$};
	
	\node[] [below =0.1cm of 1] {\textcolor{gray}{$0$}};
	\node[] [above =0.1cm of 2]  {\textcolor{gray}{$-2$}};
	\node[] [below =0.1cm of 3] {\textcolor{gray}{$-4$}};
	\node[] [above =0.1cm of 4] {\textcolor{gray}{$6$}};
		
	\path[draw,thick]
	(1) edge node [left]{$\omega_{ab}$} (2)
	(2) edge node [above]{$\omega_{bd}$} (4)
	(1) edge node [below]{$\omega_{ac}$} (3)
	(4) edge node [right]{$\omega_{cd}$} (3);
	\end{tikzpicture}
}
& 
\includegraphics[clip=true, trim=5cm 10cm 5cm 10cm, width=0.4\textwidth]{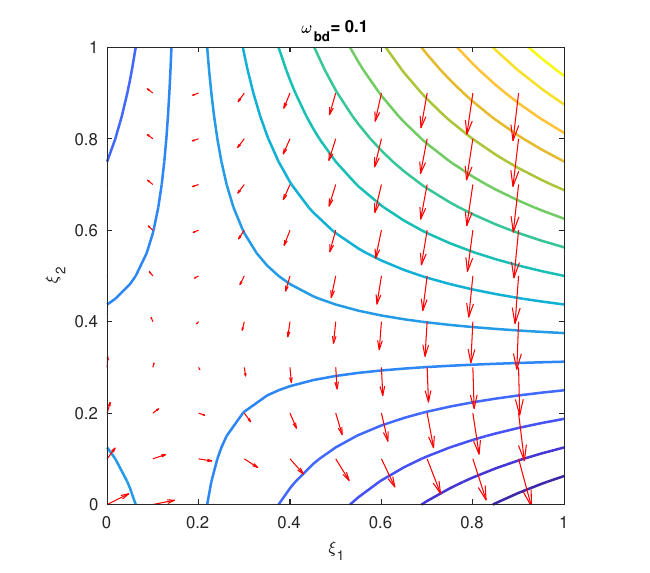} \\
\includegraphics[clip=true, trim=5cm 10cm 5cm 10cm, width=0.4\textwidth]{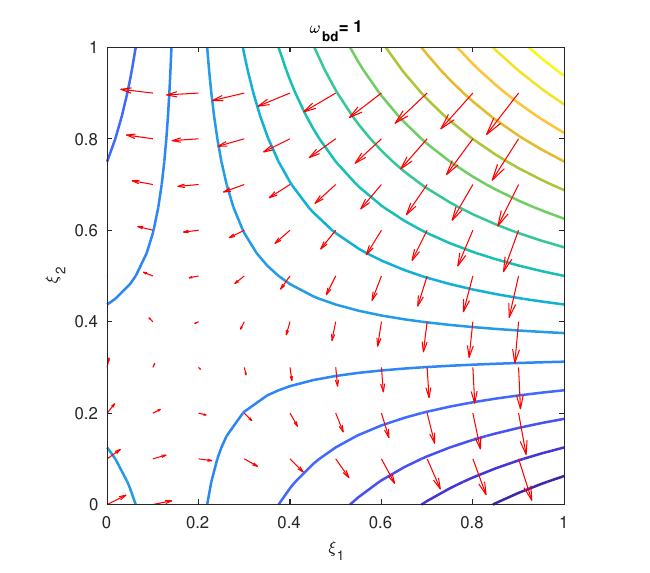}
& 
\includegraphics[clip=true, trim=5cm 10cm 5cm 10cm, width=0.4\textwidth]{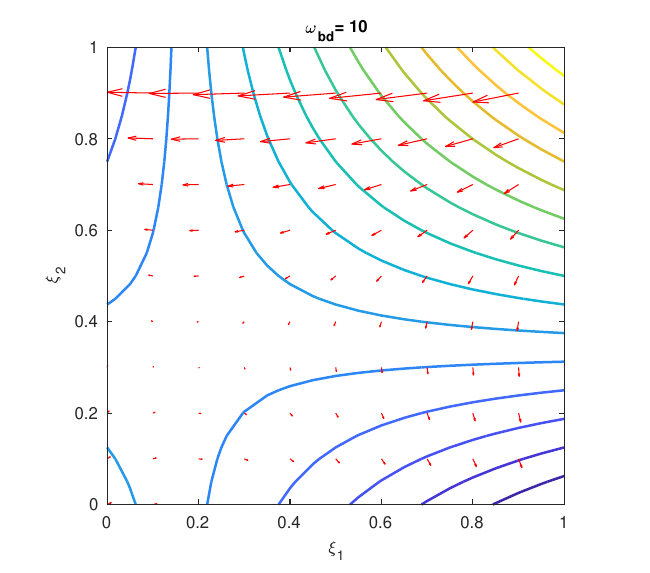}
\end{tabular}
\caption{Negative Wasserstein gradient on the parameter space $[0,1]^2$ of the two-bit independence model from Example~\ref{example:independencemodel}. 
We fix the state graph shown on the top left, and a function $f$ with values shown in gray next to the state nodes. 
We evaluate the gradient flow for three different choices of the graph weight $\omega_{bd}$. 
When the weight $\omega_{bd}$ is small, the flow from $d$ towards $b$ (a local minimum) is suppressed. 
A large weight has the opposite effect. 
The contours are for the objective function $F( p(\xi)) = \mathbb{E}_{ p(\xi)}[f]$. 
}
\label{fig2}
\end{figure}

In Figure~\ref{fig2} we plot the negative Wasserstein gradient vector field in the parameter space $\Theta=[0,1]^2$. As can be seen, the Wasserstein gradient direction depends on the ground metric on sample space (encoded in the edge weights). 
If $b$ and $d$ are far away, there is higher tendency to go $c$, rather than $b$. This reflects the intuition that, the more ground distance between $b$ and $d$, 
the harder for the probability distribution to move from its concentration place $b$ to $d$. 
We observe that the the attraction region of the two local minimizers changes dramatically as the ground metric between $b$ and $d$ changes, i.e., as  $\omega_{bd}$ varies from $0.1$, $1$, $10$. This is different in the Fisher-Rao gradient flow, plotted in Figure~\ref{fig2a}, which is independent of the ground metric on sample space.

The above result illustrates the displacement convexity shown in Theorem \ref{DC}. Different ground metric exhibits different displacement convexity of $f$ on parameter space $(\Theta, g)$. These properties lead to different convergence regions of Wasserstein gradient flows. 

\begin{figure}[h]
	\centering
\includegraphics[clip=true, trim=5cm 10cm 5cm 10cm, width=0.4\textwidth]{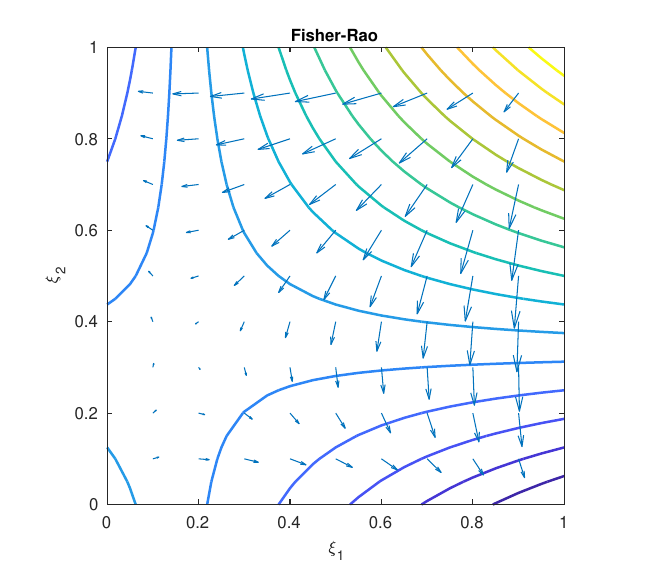}
\caption{Fisher-Rao gradient vector field for the same objective function of Figure~\ref{fig2}.}
\label{fig2a}
\end{figure}

\end{example}

\begin{example}[Wasserstein gradient for maximum likelihood estimation]
In maximum likelihood estimation, we seek to minimize the Kullback-Leibler divergence 
$$
\operatorname{KL}(q \| p(\theta))=\sum_{x\in I}q_x\log\frac{q_x}{p_x(\theta)},
$$ 
where $q$ is the empirical distribution of some given data. The Wasserstein gradient flow of $\operatorname{KL}(q \| p(\theta))$ satisfies 
\begin{equation*}
\frac{d\theta}{dt}=\Big(J_\theta  p(\theta)^{\ts}L( p(\theta))^{\dd}J_\theta  p(\theta)\Big)^{\dd}J_\theta  p(\theta)^{\ts} 
\left(\frac{q}{ p(\theta)}\right). 
\end{equation*}

In this example we consider hierarchical log-linear models as our parametrized probability models, which are an important type of exponential families describing interactions among groups of random variables. 
Concretely, for an inclusion closed set $S$ of subsets of $\{1,\ldots, n\}$, the hierarchical model $\mathcal{E}_S$ for $n$ binary variables is the set of distributions of the form 
$$
p_x(\theta) = \frac{1}{Z(\theta)}\exp\Big(\sum_{\lambda\in S} \theta_\lambda \phi_\lambda(x)\Big), \quad x\in\{0,1\}^n, 
$$ 
for all possible choices of parameters $\theta_\lambda\in\mathbb{R}$, $\lambda\in S$. 
Here the $\phi_\lambda$ are real valued functions with $\phi_\lambda(x)=\phi_\lambda(y)$ whenever $x_i=y_i$ for all $i\in\lambda$. 
We consider two different choices of 
$\phi_\lambda$, $\lambda\in S$, corresponding to two different parametrizations of the model. 
\begin{itemize}
	\item 
	Our first choice are the orthogonal characters 
	$$
	\sigma_\lambda(x) = \prod_{i\in\lambda} (-1)^{x_i} = e^{ i \pi  \langle 1_\lambda , x\rangle }, \quad x\in\{0,1\}^n, 
	$$
	which can be interpreted as a Fourier basis for the space of real valued functions over binary vectors. 
	\item 
	As an alternative choice we consider the basis of monomials 
	$$
	\pi_\lambda(x) = \prod_{i\in\lambda} x_i, \quad x\in\{0,1\}^n, 
	$$
	which is not orthogonal, but is frequently used in practice. 
\end{itemize}
When $S=\{ \lambda\subseteq\{1,\ldots, n\}\colon |\lambda|\leq k\}$, the model is called $k$-interaction model. 
We consider $k$-interaction models with $k=1,\ldots, n$ (independence model, pair interaction model, three way interaction model, etc.), with the two parametrizations, $\sigma$ (orthogonal sufficient statistics) and $\pi$ (non-orthogonal sufficient statistics).

We compare the Euclidean, Fisher, and Wasserstein gradients. 
For binary variables, the Hamming distance is a natural ground metric notion. 
Accordingly, we define the Wasserstein metric with the uniformly weighted graph of the binary cube. 
We sampled a few target distributions on $\{0,1\}^n$ uniformly at random (uniform Dirichlet).  
For each target distribution, we initialize the model at the uniform distribution, $\theta_0=0$. 
The gradient descent parameter iteration is 
$$
\theta_{t+1} = \theta_t - \gamma_t G(\theta_t)^{-1} \nabla \operatorname{KL}(q\|p_{\theta_t}), 
$$
where $G$ is the corresponding metric (Euclidean, Fisher, or Wasserstein), $\nabla$ is the standard gradient operator with respect to the model parameter $\theta$, and $\gamma_t\in\mathbb{R}_+$ is the learning rate (step size).   
The choice of the learning rate $\gamma_t$ is important and the optimal value may vary for different methods and problems. 
We implemented an adaptive method to handle this as follows. 
We set an initial learning rate $\gamma_0=0.001$, and at each iteration $t$, if the divergence does not decrease, we scale down the learning rate by a factor of $3/4$. We also tried a few other methods, including backtracking line search and Adam~\cite{journals/corr/KingmaB14}, which is a method based on adaptive estimates of lower-order moments of the gradient. The stopping criterion was that the infinity norm of the expectation parameter matched the data expectation parameter to within 1 percent.

The results are shown in Figures~\ref{fig:RobustAdam}. 
The convergence to the final value can be monitored in terms of the normalized area under the optimization curve, $\sum_{t=1}^T (D_t-D_T)/(D_0-D_T)$, where $D_t$ is the divergence value at iteration $t$, and $T$ is the final time. 
All methods achieved similar values of the divergence, except for the Euclidean gradient with non-orthogonal parametrization, which did not always reach the minimum. 
For the Fisher and Wasserstein gradients, the learning paths were virtually identical under the two different model parametrizations, as we already expected from the fact that these are covariant gradients. 
On the other hand, for the Euclidean gradient, the paths (and the number of iterations) were heavily dependent on the model parametrization, with the orthogonal basis usually being a much better choice than the non-orthogonal basis. 
In terms of the number of iterations until the convergence criterion was satisfied, 
the comparison is difficult because different methods work best with different step sizes. 
With the simple adaptive method and a suitable initial step size, the Wasserstein gradient was faster than the Euclidean and Fisher gradients. On the other hand, using Adam to adapt the step size, orthogonal Euclidean, Fisher, and Wasserstein were comparable. 

\begin{figure}[h]
	\centering
	\includegraphics[clip=true, trim=1.6cm 9cm 2cm 10cm,width=.7\textwidth]{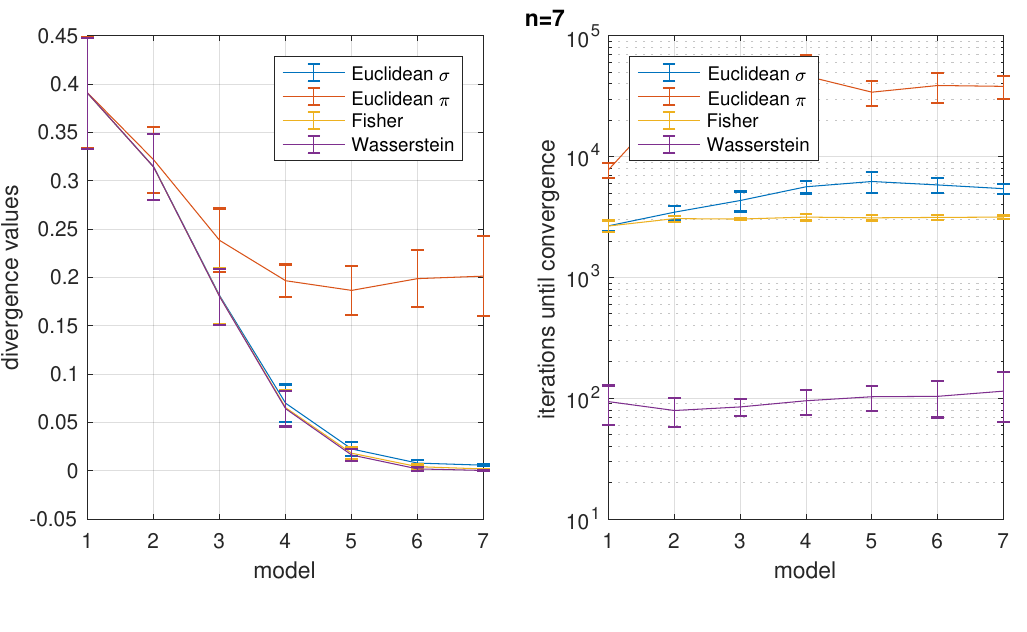}
	\includegraphics[clip=true, trim=5.5cm 7cm 6cm 7.4cm,width=.29\textwidth]{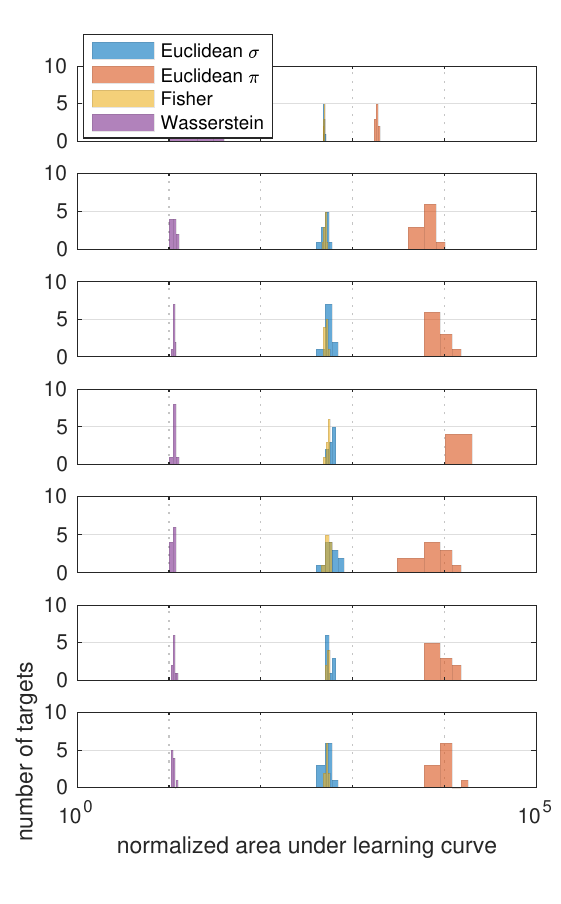}\\
	\includegraphics[clip=true, trim=1.6cm 9cm 2cm 10cm,width=.7\textwidth]{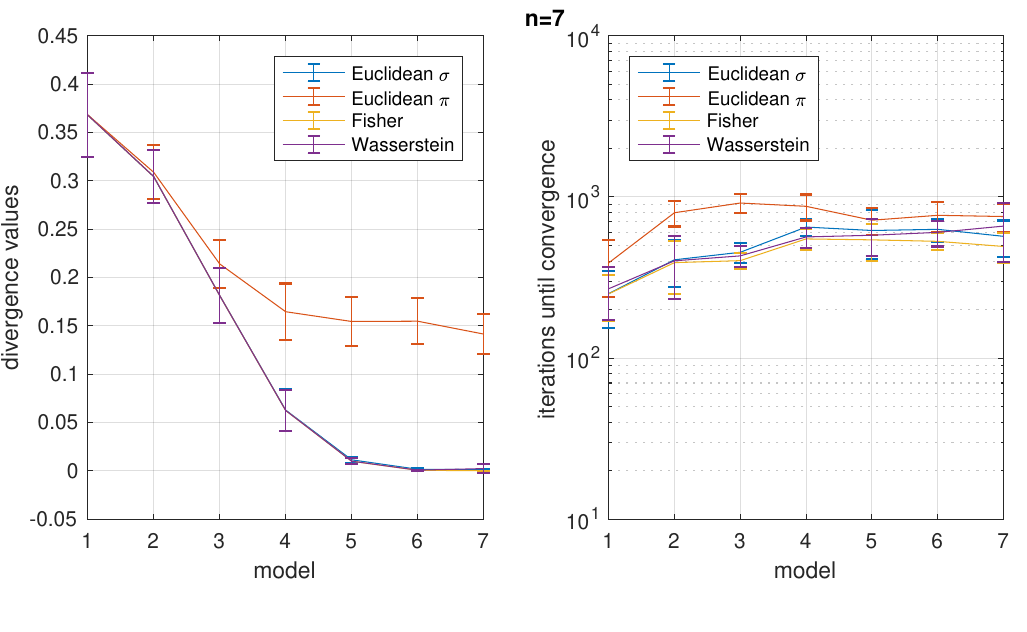}
	\includegraphics[clip=true, trim=5.5cm 7cm 6cm 7.4cm,width=.29\textwidth]{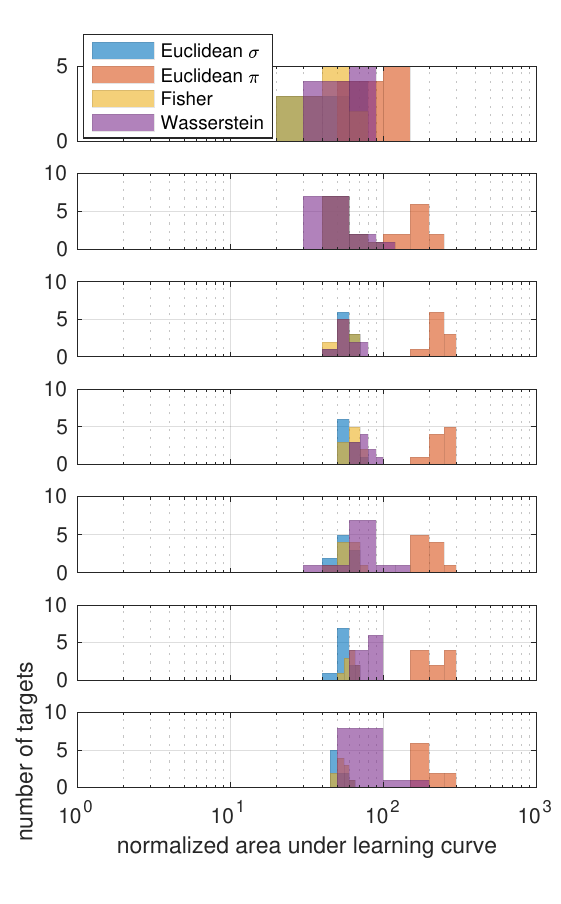}
	\caption{
		Divergence minimization for random target distributions on $\{0,1\}^n$, $n=7$, over $k$-interaction models with $k=1,\ldots,n$. 
		Shown is the average value of the divergence after optimization by Euclidean, Fisher, and Wasserstein gradient descent, 
		and the corresponding number of gradient iterations. 
Orthogonal and non-orthogonal parametrization are indicated by $\sigma$ and $\pi$. 
		The right hand side shows histograms of the normalized area under the optimization curves. 
		The top figures are using a simple adaptive method for selecting the step size, and the bottom figures are using Adam. 
	}
	\label{fig:RobustAdam}
\end{figure}
\end{example}


\section{Discussion}
We introduced the Wasserstein statistical manifolds, which are submanifolds of the probability simplex with the $L^2$-Wasserstein Riemannian metric tensor. With this, we defined an optimal transport natural gradient flow on parameter space. 

The Wasserstein distance has already been discussed with divergences in information geometry and also shown to be useful in machine learning, for instance in training restricted Boltzmann machines and generative adversarial networks. 
In this work, we used the Wasserstein distance to define a geometry on the parameter space of a statistical model. Following this geometry, we establish a corresponding natural gradient and displacement convexity on parameter space.

We presented an application of the Wasserstein natural gradient method to maximum likelihood estimation in hierarchical probability models. The experiments show that, in combination with a suitable step size, the Wasserstein gradient can be a competitive optimization method and even reduce the required number of parameter iterations compared both to Euclidean and Fisher gradient methods. It will be essential to conduct further experimental studies to better understand the effects of the learning rate, as well as the interplay of ground metric, model, and optimization problem. In our current implementation, the Wasserstein gradient involved heavier computational costs compared to the Euclidean and Fisher gradients. For applications, it will be important to explore efficient computation and approximation approaches. 

Regarding the theory, we suggest that many studies from information geometry will have a natural analog or extension in the Wasserstein statistical manifold. Some questions to consider include the following. Is it possible to characterize the Wasserstein metric on probability manifolds through an invariance requirement of Chentsov type? 
For instance, the work~\cite{e16063207} formulates extensions of Markov embeddings for polytopes and weighted point configurations. 
Is there a weighted graph structure for which the corresponding Wasserstein metric recovers the Fisher metric? 

The critical innovation coming from the Wasserstein gradient in comparison to the Fisher gradient is that it incorporates a ground metric in sample space. We suggest that this could have a positive effect not only concerning optimization, as discussed above, but also regarding generalization performance, in interplay with the optimization. The reason is that the ground metric on sample space provides means to introduce preferences in the hypothesis space. 
The specific form of such a regularization still needs to be developed and investigated. In this regard, a natural question is how to define natural ground metric notions. These could be fixed in advance or trained. 

We hope that this paper contributes to strengthening the emerging interactions between information geometry and optimal transport, in particular, to machine learning problems, and to develop better natural gradient methods. 

\noindent\textbf{Acknowledgement}
The authors would like to thank Prof.~Luigi Malag\`o for his inspiring talk at UCLA in December 2017. 
This project has received funding from the European Research Council (ERC) under the European Union's Horizon 2020 research and innovation programme (grant agreement n\textsuperscript{o} 757983). 

\bibliographystyle{abbrv}

\begin{thebibliography}{10}

\bibitem{NIPS1996_1248}
S.~Amari.
\newblock Neural learning in structured parameter spaces - natural {R}iemannian
  gradient.
\newblock In M.~C. Mozer, M.~I. Jordan, and T.~Petsche, editors, {\em Advances
  in Neural Information Processing Systems 9}, pages 127--133. MIT Press, 1997.

\bibitem{NG}
S.~Amari.
\newblock Natural gradient works efficiently in learning.
\newblock {\em Neural Computation}, 10(2):251--276, 1998.

\bibitem{IG}
S.~Amari.
\newblock {\em Information Geometry and Its Applications}.
\newblock Number volume 194 in Applied mathematical sciences. {Springer},
  Japan, 2016.

\bibitem{LP}
S.~Amari, R.~Karakida, and M.~Oizumi.
\newblock Information geometry connecting {W}asserstein distance and
  {K}ullback-{L}eibler divergence via the {{Entropy}}-{{Relaxed Transportation
  Problem}}.
\newblock {\em arXiv:1709.10219 [cs, math]}, 2017.

\bibitem{WGAN}
M.~Arjovsky, S.~Chintala, and L.~Bottou.
\newblock Wasserstein {{GAN}}.
\newblock {\em arXiv:1701.07875 [cs, stat]}, 2017.

\bibitem{IG2}
N.~Ay, J.~Jost, H.~L{\^e}, and L.~Schwachh{\"o}fer.
\newblock {\em Information Geometry}.
\newblock Ergebnisse der Mathematik und ihrer Grenzgebiete. 3. Folge / A Series
  of Modern Surveys in Mathematics. Springer International Publishing, 2017.

\bibitem{BE}
D.~Bakry and M.~{\'E}mery.
\newblock Diffusions hypercontractives.
\newblock In J.~Az{\'e}ma and M.~Yor, editors, {\em S{\'e}minaire de
  Probabilit{\'e}s XIX 1983/84}, pages 177--206, Berlin, Heidelberg, 1985.
  Springer Berlin Heidelberg.

\bibitem{Benamou2000}
J.-D. Benamou and Y.~Brenier.
\newblock A computational fluid mechanics solution to the
  {{Monge}}-{{Kantorovich}} mass transfer problem.
\newblock {\em Numerische Mathematik}, 84(3):375--393, 2000.

\bibitem{Campbell}
L.~Campbell.
\newblock An extended {{\v C}encov} characterization of the information metric.
\newblock {\em Proceedings of the American Mathematical Society}, 98:135--141,
  1986.

\bibitem{Gangbo2}
E.~A. Carlen and W.~Gangbo.
\newblock Constrained {{Steepest Descent}} in the 2-{{Wasserstein Metric}}.
\newblock {\em Annals of Mathematics}, 157(3):807--846, 2003.

\bibitem{cencov1982}
N.~N. {\v{C}}encov.
\newblock {\em Statistical decision rules and optimal inference}, volume~53 of
  {\em Translations of Mathematical Monographs}.
\newblock American Mathematical Society, Providence, R.I., 1982.
\newblock Translation from the Russian edited by Lev J. Leifman.

\bibitem{chow2012}
S.-N. Chow, W.~Huang, Y.~Li, and H.~Zhou.
\newblock Fokker\textendash{}{{Planck Equations}} for a {{Free Energy
  Functional}} or {{Markov Process}} on a {{Graph}}.
\newblock {\em Archive for Rational Mechanics and Analysis}, 203(3):969--1008,
  2012.

\bibitem{Li2}
S.-N. Chow, W.~Li, and H.~Zhou.
\newblock A discrete {{Schrodinger}} equation via optimal transport on graphs.
\newblock {\em arXiv:1705.07583 [math]}, 2017.

\bibitem{Li1}
S.-N. Chow, W.~Li, and H.~Zhou.
\newblock Entropy dissipation of Fokker-Planck equations on graphs.
\newblock {\em Discrete and Continuous Dynamical Systems - A}, 38(10),4929--4950, 2018.

\bibitem{Graph}
F.~R.~K. Chung.
\newblock {\em Spectral Graph Theory}.
\newblock Number no. 92 in Regional conference series in mathematics.
  {Published for the Conference Board of the mathematical sciences by the
  American Mathematical Society}, Providence, R.I, 1997.

\bibitem{LWL}
C.~Frogner, C.~Zhang, H.~Mobahi, M.~Araya-Polo, and T.~Poggio.
\newblock Learning with a {{Wasserstein Loss}}.
\newblock {\em arXiv:1506.05439 [cs, stat]}, 2015.

\bibitem{Li3}
W.~Gangbo, W.~Li, and C.~Mou.
\newblock Geodesic of minimal length in the set of probability measures on
  graphs.
\newblock {\em accepted in ESAIM: COCV}, 2018.

\bibitem{IGWD}
R.~Karakida and S.~Amari.
\newblock Information geometry of wasserstein divergence.
\newblock In F.~Nielsen and F.~Barbaresco, editors, {\em Geometric Science of
  Information}, pages 119--126, Cham, 2017. Springer International Publishing.

\bibitem{journals/corr/KingmaB14}
D.~P. Kingma and J.~Ba.
\newblock Adam: A method for stochastic optimization.
\newblock {\em CoRR}, abs/1412.6980, 2014.

\bibitem{Lafferty}
J.~D. Lafferty.
\newblock The density manifold and configuration space quantization.
\newblock {\em Transactions of the American Mathematical Society},
  305(2):699--741, 1988.

\bibitem{Lebanon05}
G.~Lebanon.
\newblock Axiomatic geometry of conditional models.
\newblock {\em IEEE Transactions on Information Theory}, 51(4):1283--1294,
  2005.

\bibitem{LiG}
W.~Li.
\newblock Geometry of probability simplex via optimal transport.
\newblock {\em arXiv:1803.06360 [math]}, 2018.

\bibitem{RLG}
W.~Li and G.~Montufar.
\newblock Ricci curvature for parameter statistics via optimal transport.
\newblock {\em arXiv:1807.07095}, 2018.

\bibitem{Li_COM}
W.~Li, P.~Yin, and S.~Osher.
\newblock Computations of {{Optimal Transport Distance}} with {{Fisher
  Information Regularization}}.
\newblock {\em Journal of Scientific Computing}, 2017.

\bibitem{Lott}
J.~Lott.
\newblock Some {{Geometric Calculations}} on {{Wasserstein Space}}.
\newblock {\em Communications in Mathematical Physics}, 277(2):423--437, 2007.

\bibitem{Maas}
J.~Maas.
\newblock Gradient flows of the entropy for finite {{Markov}} chains.
\newblock {\em Journal of Functional Analysis}, 261(8):2250--2292, 2011.

\bibitem{Malago:2011:TGE:1967654.1967675}
L.~Malag\`{o}, M.~Matteucci, and G.~Pistone.
\newblock Towards the geometry of estimation of distribution algorithms based
  on the exponential family.
\newblock In {\em Proceedings of the 11th Workshop Proceedings on Foundations
  of Genetic Algorithms}, FOGA '11, pages 230--242, New York, NY, USA, 2011.
  ACM.

\bibitem{Malago2}
L.~Malag{\`o} and G.~Pistone.
\newblock Natural {{Gradient Flow}} in the {{Mixture Geometry}} of a {{Discrete
  Exponential Family}}.
\newblock {\em Entropy}, 17(12):4215--4254, 2015.

\bibitem{M}
A.~Mielke.
\newblock A gradient structure for reaction\textendash{}diffusion systems and
  for energy-drift-diffusion systems.
\newblock {\em Nonlinearity}, 24(4):1329--1346, 2011.

\bibitem{IGW}
K.~Modin.
\newblock Geometry of {{Matrix Decompositions Seen Through Optimal Transport}}
  and {{Information Geometry}}.
\newblock {\em Journal of Geometric Mechanics}, 9(3):335--390, 2017.

\bibitem{Boltzman}
G.~Montavon, K.-R. M{\"u}ller, and M.~Cuturi.
\newblock Wasserstein {{Training}} of {{Restricted Boltzmann Machines}}.
\newblock In D.~D. Lee, M.~Sugiyama, U.~V. Luxburg, I.~Guyon, and R.~Garnett,
  editors, {\em Advances in {{Neural Information Processing Systems}} 29},
  pages 3718--3726. {Curran Associates, Inc.}, 2016.

\bibitem{e16063207}
G.~Mont\'ufar, J.~Rauh, and N.~Ay.
\newblock On the {F}isher metric of conditional probability polytopes.
\newblock {\em Entropy}, 16(6):3207--3233, 2014.

\bibitem{qf}
E.~Nelson.
\newblock {\em Quantum Fluctuations}.
\newblock Princeton series in physics. {Princeton University Press}, Princeton,
  N.J, 1985.

\bibitem{otto2001}
F.~Otto.
\newblock The geometry of dissipative evolution equations: The porous medium
  equation.
\newblock {\em Communications in Partial Differential Equations},
  26(1-2):101--174, 2001.

\bibitem{Pascanu+Bengio-ICLR2014}
R.~Pascanu and Y.~Bengio.
\newblock Revisiting natural gradient for deep networks.
\newblock In {\em International Conference on Learning Representations 2014
  (Conference Track)}, Apr. 2014.

\bibitem{10.1007/11564096_29}
J.~Peters, S.~Vijayakumar, and S.~Schaal.
\newblock Natural actor-critic.
\newblock In J.~Gama, R.~Camacho, P.~B. Brazdil, A.~M. Jorge, and L.~Torgo,
  editors, {\em Machine Learning: ECML 2005}, pages 280--291, Berlin,
  Heidelberg, 2005. Springer Berlin Heidelberg.

\bibitem{GW}
A.~Takatsu.
\newblock Wasserstein geometry of {{Gaussian}} measures.
\newblock {\em Osaka Journal of Mathematics}, 48(4):1005--1026, 2011.

\bibitem{vil2008}
C.~Villani.
\newblock {\em Optimal Transport: Old and New}.
\newblock Number 338 in Grundlehren der mathematischen Wissenschaften.
  {Springer}, Berlin, 2009.

\bibitem{Wong}
T.-K. Wong.
\newblock Logarithmic divergences from optimal transport and {R}\'enyi
  geometry.
\newblock {\em arXiv:1712.03610 [cs, math, stat]}, 2017.

\bibitem{Yi:2009:SSU:1553374.1553522}
S.~Yi, D.~Wierstra, T.~Schaul, and J.~Schmidhuber.
\newblock Stochastic search using the natural gradient.
\newblock In {\em Proceedings of the 26th Annual International Conference on
  Machine Learning}, ICML '09, pages 1161--1168, New York, NY, USA, 2009. ACM.

\end{thebibliography}

\appendix
\section*{Appendix}

In this appendix we review the equivalence of static and dynamical formulations of the $L^2$-Wasserstein metric formally. For more details see \cite{vil2008}. 

Consider the duality of linear programming. 
\begin{equation}\label{aa}
\begin{split}
\frac{1}{2}W(\rho^0,\rho^1)^2=&\inf_{\pi\geq 0}\Big\{\int_{\Omega}\int_{\Omega}\frac{1}{2}d_{\Omega}(x,y)^2\pi(x,y)dxdy\colon \int_\Omega\pi dy=\rho^0(x),~\int_\Omega\pi dx=\rho^1(y)\Big\}\\
=&\sup_{\Phi^1, \Phi^0}\Big\{\int_{\Omega}\Phi^1(y)\rho^1(y)dy-\int_\Omega\Phi^0(x)\rho^0(x)dx\colon \Phi^1(y)-\Phi^1(x)\leq \frac{1}{2}d_\Omega(x,y)^2\Big\}. 
\end{split}
\end{equation}
 By standard considerations, the supremum in the last formula is attained when  
 \begin{equation}\label{HL}
 \Phi^1(y)=\sup_{x\in \Omega}~\Phi^0(x)+\frac{1}{2}d_\Omega(x,y)^2.
 \end{equation}
This means that $\Phi^1$, $\Phi^0$ are related to the viscosity solution of the Hamilton-Jacobi equation on $\Omega$:
\begin{equation}\label{HJB}
\frac{\partial\Phi(t,x)}{\partial t}+\frac{1}{2}g_x^\Omega(\nabla\Phi(t,x), \nabla\Phi(t,x))=0,
\end{equation}
with $\Phi^0(x)=\Phi(0,x)$, $\Phi^1(x)=\Phi(1,x)$. Hence \eqref{aa} becomes 
\begin{equation*}
\frac{1}{2}W(\rho^0,\rho^1)^2=\sup_{\Phi}\Big\{\int_{\Omega}\Phi^1(x)\rho^1(x)-\Phi^0(x)\rho^0(x)dx\colon \frac{\partial\Phi(t,x)}{\partial t}+\frac{1}{2}g_x^\Omega(\nabla\Phi(t,x), \nabla\Phi(t,x))=0 \Big\}.
\end{equation*}
By the duality of above formulas, we can obtain variational problem \eqref{def_metric}.
 In other words, consider the dual variable of $\Phi_t=\Phi(t,x)$ by the density path $\rho_t=\rho(t,x)$, then 
\begin{equation*}
\begin{split}
&\frac{1}{2}W(\rho^0,\rho^1)^2\\
=&\sup_{\Phi_t}\inf_{\rho_t}~\int_{\Omega}\Phi^1\rho^1-\Phi^0\rho^0dx-\int_0^1\int_{\Omega}\rho_t\big[ \partial_t\Phi_t+\frac{1}{2}g_x^\Omega(\nabla\Phi_t, \nabla\Phi_t)dx\big] dt\\
=&\sup_{\Phi_t}\inf_{\rho_t}~\int_{\Omega}\Phi^1\rho^1-\Phi^0\rho^0dx-\int_0^1\int_{\Omega}\rho_t \partial_t\Phi_tdxdt- \int_0^1\int_{\Omega}\frac{1}{2}g_x^\Omega(\nabla\Phi_t, \nabla\Phi_t)\rho_tdx dt\\
=&\sup_{\Phi_t}\inf_{\rho_t}~\int_0^1\int_{\Omega}\partial_t\rho_t \Phi_t-g_x^\Omega(\nabla\Phi_t, \nabla\Phi_t)\rho_tdx dt+\int_0^1\int_{\Omega}\frac{1}{2}g_x^\Omega(\nabla\Phi_t, \nabla\Phi_t)\rho_tdx dt \\
=&\inf_{\rho_t}\sup_{\Phi_t}~\int_0^1\int_{\Omega}\Phi_t(\partial_t\rho_t+\textrm{div}(\rho\nabla\Phi_t)) dt+\int_0^1\int_{\Omega}\frac{1}{2}g_x^\Omega(\nabla\Phi_t, \nabla\Phi_t)\rho_tdx dt \\
=&\inf_{\rho_t}~\Big\{\int_0^1\int_{\Omega}\frac{1}{2}g_x^\Omega(\nabla\Phi_t, \nabla\Phi_t)\rho_tdx dt\colon \partial_t\rho_t+\textrm{div}(\rho\nabla\Phi_t)=0,~\rho_0=\rho^0, ~\rho_1=\rho^1\Big\}.
\end{split}
\end{equation*}
The third equality is derived by integration by parts w.r.t.~$t$ and the fourth equality is by switching infimum and supremum relations and integration by parts w.r.t.~$x$. 

In the above derivations, the relation of Hopf-Lax formula~\eqref{HL} and Hamilton-Jacobi equation~\eqref{HJB} plays a key role for the equivalence of static and dynamic formulations of the Wasserstein metric. 
This is also a consequence of the fact that the sample space $\Omega$ is a length space, i.e.,  
\begin{equation*}
d_\Omega(x,y)^2=\inf_{\gamma(t)}\Big\{\int_0^1g_{\gamma(t)}^\Omega(\dot \gamma, \dot\gamma)dt\colon \gamma(0)=x,~\gamma(1)=y\Big\}.
\end{equation*}
However, in a discrete sample space $I$, there is no path $\gamma(t)\in I$ connecting two discrete points. 
Thus the relation between \eqref{HL} and \eqref{HJB} does not hold on $I$. This indicates that  in discrete sample spaces, the Wasserstein metric in Definition~\eqref{def_metric} can be different from the one defined by linear programming~\eqref{linear}. See many related discussions in~\cite{chow2012, Maas}.

\section*{Notations}
We use the following notations.

\begin{tabular}{|l|c|c|}
\hline 
Continuous/Discrete sample space &$\Omega$ & $I$ \\
\hline 
Inner product  &$g^\Omega$ & $g^I$\\
Gradient& $\nabla$ & $\nabla_G$\\
 divergence & $\textrm{div}$ & $\textrm{div}_G$\\
Hessian in $\Omega$ &\textrm{Hess} & \\
Potential function set&$\mathcal{F}(\Omega)$ & $\mathcal{F}(I)$\\
Weighted Laplacian operator & $-\nabla\cdot(\rho\nabla)$ & $L(p)$\\
\hline
\end{tabular}

\begin{tabular}{|l|c|c|}
\hline 
Continuous/Discrete probability space &
$\mathcal{P}_+(\Omega)$ & 
$\mathcal{P}_+(I)$\\
\hline 
Probability distribution & $\rho$ & $p$\\
Tangent space & $T_\rho\mathcal{P}_+(\Omega)$ &  $T_p\mathcal{P}_+(I)$\\
Wasserstein metric tensor& $g^W$ & $g^W$\\
Dual coordinates &$\Phi(x)$& $(\Phi_i)_{i=1}^n$\\
Primal coordinates & $\sigma(x)$ & $(\sigma_i)_{i=1}^n$ \\
First differential operator & $\delta_\rho$& $\nabla_p$\\
 Second differential operator &$\delta^2_{\rho\rho}$ & 
 \\
Gradient operator && $\nabla_W$\\
Hessian operator &&$\textrm{Hess}_W$ \\
Levi-Civita connection & & $\nabla^W_{\cdot}\cdot$\\
\hline
\end{tabular}

\begin{tabular}{|l|c|c|}	
\hline
Parameter space/Probability model &$\Theta$ & $p(\Theta)$\\
\hline 
Inner product  &$g_\theta$ & $g_{p(\theta)}$\\
Tangent space & $T_\theta\Theta$ &  $T_{p(\theta)}p(\Theta)$\\
$L^2$-Wasserstein matrix  & $G(\theta)$ & \\
$L^2$-Wasserstein distance  &$\textrm{Dist}$& $\textrm{Dist}$\\
Second fundamental form  & & $B(\cdot, \cdot)$\\
Projection operator  & & $H$\\
Levi-Civita connection & &$(\nabla^W_\cdot\cdot)^{||}$\\
Jacobi operator  & $J_\theta $ &\\
First differential operator & $\nabla_\theta$ & \\
Gradient operator &$\nabla_g$ & \\
Hessian operator & $\textrm{Hess}_g$ & \\
\hline
\end{tabular}

\end{document}